\documentclass[notitlepage, 11pt, a4paper]{article}

\usepackage[utf8]{inputenc}
\usepackage[T1]{fontenc}
\usepackage[natbib=true, backend=biber, style=numeric, maxbibnames=9]{biblatex}
\usepackage{csquotes}
\usepackage{amsmath}
\usepackage{amsfonts}
\usepackage{amssymb}
\usepackage{amsthm}
\usepackage[affil-it]{authblk}
\usepackage[normalem]{ulem}
\usepackage{dsfont}
\usepackage{float}
\usepackage{times}
\usepackage{bbm}
\usepackage{mathtools}
\usepackage{caption}
\usepackage{mathrsfs}
\usepackage{comment}
\usepackage{bm}
\usepackage{enumerate}
\usepackage{lipsum}
\usepackage{titling}
\usepackage{romannum}
\usepackage[colorlinks,linkcolor=blue,citecolor=blue,urlcolor=blue]{hyperref}
\usepackage[capitalise]{cleveref}
\usepackage[left=3cm, right=3cm, bottom=2cm, top=2cm]{geometry}
\usepackage[toc,page]{appendix}

\newcommand*\diff{\mathop{}\!\mathrm{d}}

\newcommand{\cadlag}{c\`adl\`ag }

\newcommand{\N}{\ensuremath{\mathbb{N}}}

\newcommand{\R}{\ensuremath{\mathbb{R}}}

\newcommand{\1}{\ensuremath{\mathds{1}}}

\newcommand\norm[1]{\left\lVert#1\right\rVert}
\newcommand\nnorm[1]{|#1|}
\newcommand\nnnorm[1]{\left|#1\right|} 

\newcommand{\prob}{\ensuremath\mathbb{P}}

\newcommand{\E}{\mathbb{E}}

\renewcommand{\geq}{\geqslant}
\renewcommand{\leq}{\leqslant}

\newcommand{\ind}{\mathbbm{1}}
\newcommand{\pref}{\ensuremath{p_\text{ref}}}

\newtheorem{theorem}{Theorem}[section]

\newtheorem{proposition}[theorem]{Proposition}
\newtheorem{corollary}[theorem]{Corollary}
\newtheorem{assumption}[theorem]{Assumption}

\theoremstyle{definition}

\theoremstyle{remark}

\numberwithin{equation}{section}
\crefname{example}{Example}{Examples}

\begin{document}
\pagenumbering{arabic}
\title{Probabilistic estimates for a system of\\noisy integrate-and-fire neurons}

\author{Ben Hambly\thanks{Mathematical Institute, University of Oxford, United Kingdom, Email: hambly@maths.ox.ac.uk},
\and Alda\"ir Petronilia\thanks{Mathematical Institute, University of Oxford, United Kingdom, Email: aldairpetronilia@gmail.com},
\and Christoph Reisinger\thanks{Mathematical Institute, University of Oxford, United Kingdom, Email: christoph.reisinger@maths.ox.ac.uk},
\and Andreas S{\o}jmark\thanks{Department of Statistics, London School of Economics, United Kingdom, Email: a.sojmark@lse.ac.uk}}
\date{\today} 

\maketitle

 \begin{abstract}
 In this note, we establish various probabilistic estimates for an interacting particle system that describes the evolution of the membrane potentials in a network of excitatory integrate-and-fire neurons, which are subject to both idiosyncratic and common noise. These estimates serve to support a separate work that studies the large population limit and the corresponding stochastic Fokker--Planck equation for the membrane potential density.
 \end{abstract}

\section{Introduction}\label{sec: INTRODUCTION TO THE SUPPLEMENTARY MATERIAL}

In this note, we study a finite particle system which provides a rigorous formulation of the dynamics for a coupled system of noisy integrate-and-fire neurons derived  in \cite{ostojic2009synchronization}. Furthermore, the system we consider incorporates several generalisations, including a gradual spike transmission and an absolute refractory period for the membrane potential of each neuron, followed by a randomised reset below the rest potential. Throughout, we take the rest potential to be zero.

We provide a series of probabilistic estimates that, broadly, fall into the following three categories. Firstly, we derive moment bounds and sub-Gaussian tails for the membrane potentials and the cumulative spike count (i.e., the number of times that the neurons fire), uniformly in the number of neurons $N$. Moreover, a change of measure then allows us to show that the probability of a neuron firing $k$ or more times in a given interval decays at a Gaussian rate in $k$. Secondly, we establish some control on the empirical measures in expectation, showing that the mass is sufficiently concentrated and decays more than linearly near the firing threshold, uniformly in $N$. Thirdly, we derive bounds for the increments of the cumulative spike count, which, as $N\rightarrow\infty$, can serve to establish its tightness and ensure the continuity of its limit points.

The purpose of these results is to support our analysis in \citep{MAIN_PAPER_ON_THE_THE_NEURONS_WHICH_IS_YET_TO_BE_COMPLETED}, where we consider the weak convergence as $N\rightarrow \infty$ and study the well-posedness of the resulting stochastic Fokker--Planck equation for the membrane potential density in the large population limit.

\subsection{The particle system}

We fix a filtered probability space $(\Omega, \mathcal{F},(\mathcal{F}_t)_{t \ge 0},\prob)$, with a filtration $(\mathcal{F}_t)_{t \ge 0}$ that satisfies the usual conditions. Furthermore, we let this space support an infinite collection of mutually independent random variables, $\{W^i,\varsigma_k^i,\xi_k^i\}_{i,k \ge 0}$, where each $W^i$ is an $\mathcal{F}_t$-standard Brownian motion, while $\{\varsigma_k^i\}_{k, i}$ and $\{\xi_k^i\}_{k,i}$ are i.i.d.~$\mathcal{F}_0$-measurable random variables on $\R$ that satisfy Assumption \ref{ass: ASSUMPTIONS FROM SOJMARK SPDE PAPER APPLIED TO THE NEUROSCIENCE SETTING} below. When the context is clear, we shall simply write $\xi$ and $\varsigma$ for $\xi_k^i$ and $\varsigma_k^i$, respectively. The particle system with $N$ neurons is then given by

\begin{equation}\label{eq: GENERALISED FINITE PARTICLE SYSTEM IN THE INTEGRATE AND FIRE NEURON MODELS}
	\left\{
\begin{array}{r@{{}={}}l}
	\diff{}X_t^{i} &\begin{array}[t]{@{}l}
		b(t,X_t^{i},\nu_t^{N},\mathfrak{f}_t^N)\ind_{\{X_t^i < 0\}}\diff{}t + \sigma(t,\,X_t^{i})\rho(t,\nu^N_t,\mathfrak{f}_t^N)\ind_{\{X_t^i < 0\}}\diff{}W_t^0\\[0.25em]
		+ \  
		\sigma(t,\,X_t^{i})\sqrt{1 \!-\! \rho^2(t,\nu^N_t,\mathfrak{f}_t^N)}\ind_{\{X_t^i < 0\}}\diff{}W_t^i - \diff \sum_{k \ge 1} \xi_k^i\ind_{[0,t]}(\tau_k^i + \varsigma_k^i), \\[0.25em]
	\end{array} \\[0.25em]
	\tau_k^{i} & \inf\{t > \tau_{k-1}^{i} + \varsigma_{k-1}^{i} \; : \; X_{t-}^{i} \ge 0\},\qquad \tau_0^{i} = 0,\\[0.25em]\nu_t^N & \frac{1}{N}\sum_{i=1}^N \delta_{X_t^{i}}\ind_{\{X_t^i < 0\}},\quad \mathfrak{f}_t^N = \int_0^t \mathfrak{K}(t-s) \diff{F_s^N}\\[0.25em]
	F_t^N & \frac{1}{N}\sum_{i=1}^N J_t^{i},\quad J_t^{i} = \sum_{k \ge 1} \ind_{[0,t]}(\tau_k^i),
\end{array}
\right.
\end{equation}
where we shall furthermore need to keep track of the (rescaled) number of resets through the following quantities
\begin{equation}\label{eq:F-DN,J-Di}
	F_t^{D,N} = \textstyle \frac{1}{N}\sum_{i=1}^N J_t^{D,i},\quad 
	J_t^{D,i} = \textstyle\sum_{k \ge 1} \ind_{[0,t]}(\tau_k^i + \varsigma_k^i).
\end{equation}

The central object of study is $\nu^N_t$, which is the empirical measure of the neurons that are not currently in a refractory state. We refer to $\tau_k^i$ as the $k$'th spike of the $i$'th neuron, and the random variable $\varsigma_k^i$ then represents the $k$'th refractory period of the $i$'th neuron, i.e., a period of time after its $k$'th spike where the neuron is held at zero and is not part of the system. Moreover, the random variable $-\xi_k^i$ gives the membrane potential of the $i$'th neuron following its $k$'th refractory period. Finally, we refer to $F^N_t$ as the (rescaled) cumulative spike count, since it returns the total number of spikes in $[0,t]$ averaged over the $N$ particles, and, associated to this, we refer to $\mathfrak{f}_t^N$ as the spike transmission rate, which models how the dynamics of the membrane potentials are affected by spikes. For further details, we refer to the introduction of \citep{MAIN_PAPER_ON_THE_THE_NEURONS_WHICH_IS_YET_TO_BE_COMPLETED}.

Throughout, we work under the following assumptions.

\begin{assumption}[Structural assumptions]\label{ass: ASSUMPTIONS FROM SOJMARK SPDE PAPER APPLIED TO THE NEUROSCIENCE SETTING}
	We assume that the following structural conditions are satisfied by the particle system \eqref{eq: GENERALISED FINITE PARTICLE SYSTEM IN THE INTEGRATE AND FIRE NEURON MODELS}, for some $\gamma, \, R > 0$:
	\begin{enumerate}[(i)]
		\item \label{ass: ASSUMPTIONS FROM SOJMARK SPDE PAPER APPLIED TO THE NEUROSCIENCE SETTING ONE}(Growth and differentiability) The map $x \mapsto b(t,x,\mu,f)$ is $\mathcal{C}^2(\R)$ and $(t,x) \mapsto \sigma(t,x)$ is $\mathcal{C}^{1,2}([0,T]\times \R)$. Moreover, there exist $C_b,C_\sigma > 0$ such that
		\begin{align*}
			&\nnorm{b(t,x,\mu,f)} \le C_b(1 + \nnorm{x}+ \langle \mu,\nnorm{\cdot}\rangle + \nnorm{f}),\quad\nnorm{\partial_x^{(n)}b(t,x,\mu,f)}\le C_b,\quad n = 1,\,2,\\
			&\nnorm{\sigma(t,x)}\le C_\sigma,\quad\nnorm{\partial_t\sigma(t,x)}\le C_\sigma,\quad \nnorm{\partial_x^{(n)}\sigma(t,x)}\le C_\sigma,\quad n = 1,\,2.
		\end{align*}
		\item \label{ass: ASSUMPTIONS FROM SOJMARK SPDE PAPER APPLIED TO THE NEUROSCIENCE SETTING TWO}(Lipschitzness) There exists $C > 0$ such that
		\begin{align*}
			\nnorm{b(t,x,\mu,f) - b(t,\tilde{x},\tilde{\mu},\tilde{f})} &\le C(\nnorm{x - \tilde{x}} + d_0(\mu,\tilde{\mu}) + \nnorm{f - \tilde{f}}),\\
			\nnorm{\sigma(t,x) - \sigma(t,\tilde{x})} &\le C(\nnorm{x - \tilde{x}}),\\
			\nnorm{\rho(t,\mu,f) - \rho(t,\tilde{\mu},\tilde{f})} &\le C(d_1(\mu,\tilde{\mu}) + \nnorm{f - \tilde{f}}),
		\end{align*}
		where
		\begin{align*}
			d_0(\mu,\,\tilde{\mu}) &= \sup\left\{\nnnorm{\langle\mu - \tilde{\mu},\, \psi\rangle}\,:\, \norm{\psi}_{\operatorname{Lip}} \le 1,\, \nnnorm{\psi(0)} \le 1\right\},\\
			d_1(\mu,\,\tilde{\mu}) &= \sup\left\{\nnnorm{\langle\mu - \tilde{\mu},\, \psi\rangle}\,:\, \norm{\psi}_{\operatorname{Lip}} \le 1,\, \norm{\psi}_{\infty} \le 1\right\}.
		\end{align*}
		\item \label{ass: ASSUMPTIONS FROM SOJMARK SPDE PAPER APPLIED TO THE NEUROSCIENCE SETTING THREE} (Non-degeneracy) $C_\sigma$ above is chosen such that $0 < C_\sigma^{-1} \le \sigma(t,x)$ and that $0 \le \rho(t,\mu,f)\le 1 - \gamma$.
		\item \label{ass: ASSUMPTIONS FROM SOJMARK SPDE PAPER APPLIED TO THE NEUROSCIENCE SETTING FIVE} (Spike transmission) The kernel $\mathfrak{K}$ is non-negative with $\norm{\mathfrak{K}}_1 = 1$ and $\mathfrak{K} \in \mathcal{W}^{1,1}_0(\R_+)$, the Sobolev space with one weak derivative in $L^1$ and zero trace.
		\item \label{ass: ASSUMPTIONS FROM SOJMARK SPDE PAPER APPLIED TO THE NEUROSCIENCE SETTING FOUR}(Random inputs)
		The random variables $\{X_0^i, W^i,\varsigma_k^i,\xi_k^i\}_{i,k \ge 0}$ are all mutually independent of each other, and the  $\{X_0^i\}_{i\ge 1}$, $\{W^i\}_{i\ge 0}$, $\{\varsigma_k^i\}_{i,k \ge 0}$, and $\{\xi_k^i\}_{i,k \ge 0}$ form i.i.d.~sequences. The common law $\nu_0$ of the starting values $X_0^i$ has a density $V_0$ in $L^2(-\infty,0)$ and
		\begin{equation*}
			\nu_0(-\infty,-\lambda) = O(e^{-\gamma \lambda^2}) \quad \text{as} \quad \lambda \to \infty.
		\end{equation*}
		The common law of the refractory periods $\varsigma_k^i$ has a density $\pref \in \mathcal{W}^{1,1}_0(\R_+)$. Writing $\mu_\mathrm{res}$ for the common law of the reset positions $-\xi^{i}_k$, we have $\operatorname{supp}(\mu_\mathrm{res}) \subseteq [-R,-\gamma]$.
	\end{enumerate}
\end{assumption}

Under these assumptions, the particle system admits a pathwise unique solution \cite[Prop.~3.1]{MAIN_PAPER_ON_THE_THE_NEURONS_WHICH_IS_YET_TO_BE_COMPLETED}. Moreover, by \cite[Lemma~3.3]{MAIN_PAPER_ON_THE_THE_NEURONS_WHICH_IS_YET_TO_BE_COMPLETED}, we have that $\varsigma_k^i$ is independent of $\tau_k^i$ for any $i$ and $k$ which we make use of in several of the proofs.

\section{Tail Estimates and Change of Measure}\label{SECT:tail_change-measure}

We can show that if the initial law of our process is sub-Gaussian, then all the other processes of interest are also sub-Gaussian. To do this, it is more convenient to work with a process that tracks each particle continuously, denoted by $Z^i$. That is, $Z^i$ does not implement the resets but instead lives on the whole space. This reformulation allows us to employ Gr\"onwall-type arguments and, in turn, will give us control over $X^i$. 

More precisely, we let $Z_t^i \coloneqq X_t^i + \sum_{k \ge 1} \xi_k^i \ind_{[0,t]}(\tau_k^i +\varsigma_k^i)$. One readily sees that the $K^{\textnormal{th}}$ time $X^i$ hits the boundary at $0$ corresponds to the first time $Z^i$ hits $\sum_{{k} = 1}^{K - 1} \xi_{{k}}^i$. Thus, $\sum_{{k} = 1}^{J_t^i - 1} \xi_{{k}}^i \leq \sup_{s \le t} Z_s^i$ and so, since each $\xi^i_k \leq R$, we have
\begin{equation}\label{eq: BOUNDS ON THE DELAYED AND RANDOMISED JUMPS BY THE Z PARTICLE}
	\sum_{{k} \ge 1} \xi_{{k}}^i\ind_{[0,t]}(\tau_{{k}}^i  +\varsigma_{{k}}^i) \le \sum_{{k} \ge 1} \xi_{{k}}^i\ind_{[0,t]}(\tau_{{k}}^i) \le \sup_{s \le t} Z_s^i +R.
\end{equation}
Furthermore, if $X^i$ has hit the boundary $K$ times by time $t$, we have $J_t^i = K$. Therefore, as $\xi_{{k}}^{{i}}$ is lower bounded by $\gamma > 0$, we have $K \le \gamma^{-1}\sum_{{k} = 1}^{K - 1} \xi_{{k}}^i + 1$, and hence, by the above reasoning, we also deduce
\begin{equation}\label{eq: BOUNDS ON THE JUMPS OF PARTICLE I AND THE FIRING FUNCTION BY THE Z PARTICLE}
	J_t^i \le \gamma^{-1} \sup_{s \le t} Z_s^i + 1 \qquad \textnormal{and} \qquad F_t^N \le \frac{1}{N\gamma} \sum_{i =1}^N \sup_{s \le t} Z_s^i + 1.
\end{equation}
Using this observation, we may deduce by a Gr\"onwall argument that $\sup_{s \le T} \nnnorm{Z_s}^p$ is finite for any $p \ge 1$. This is the content of the following proposition.

\begin{proposition}\label{prop: BOUNDS ON THE SUP OF THE P TH MOMENTS OF THE SUP NORM OF Z}
	For any $p \ge 1$, there exists a constant $C > 0$ that depends on $p,\,b,\,\mathfrak{K},\,\sigma$ and $T$ but is independent of $N$ and $i$ such that $\E\left[\sup_{s \le T} \nnnorm{Z_s^i}^p\right] \le C$.
\end{proposition}

\begin{proof}
	By the definition of $Z^i$ above and  \eqref{eq: BOUNDS ON THE DELAYED AND RANDOMISED JUMPS BY THE Z PARTICLE}, we have
	\[ \displaystyle \nnnorm{X_s^i} \le R + 2 \sup_{u \le s} \nnnorm{Z_u^i} ,
	\quad \displaystyle \frac{1}{N} \sum_{i = 1}^N \nnnorm{X_s^i} \le R + \frac{2}{N} \sum_{i =1}^N \sup_{u \le s} \nnnorm{Z_u^i}.
	\]
    Moreover, integrating by parts in the definition of $\mathfrak{f}_s^N$, using $\mathfrak{K} \in \mathcal{W}^{1,1}_0(\R_+)$, it follows from \eqref{eq: BOUNDS ON THE JUMPS OF PARTICLE I AND THE FIRING FUNCTION BY THE Z PARTICLE} that
    \[
    \displaystyle \nnnorm{\mathfrak{f}_s^N} \le \frac{\norm{\mathfrak{K}^\prime}_1}{N\gamma} \sum_{i =1}^N \sup_{u \le s} \nnnorm{Z_u^i} + {\norm{\mathfrak{K}^\prime}}_1.
    \]
	Setting $c_1 = \max \{\gamma^{-1}{\norm{\mathfrak{K}^\prime}_1} + 2, {\norm{\mathfrak{K}^\prime}_1} + 1,2\}C_b$, by the linear growth condition on $b$,
	\begin{align}\label{eq: UPPER BOUND ON B IN TERMS OF Z AND ITS EMPIRICAL AVERAGE}
		\begin{split}
			\nnorm{b(s,\,X_s^i,\nu_s^N,\mathfrak{f}_s^N)} &\le C_b\left(1 + \nnnorm{X_s^i} + \langle\nu^N_s,\nnnorm{\cdot}\rangle + \nnnorm{\mathfrak{f}_s^N}\right)\\
			&\le c_1\left(1 + \sup_{u \le s} \nnnorm{Z_u^i} + \frac{1}{N} \sum_{j =1}^N \sup_{u \le s} \nnnorm{Z_u^j}\right).
		\end{split}
	\end{align}
	By taking norms in \eqref{eq: GENERALISED FINITE PARTICLE SYSTEM IN THE INTEGRATE AND FIRE NEURON MODELS} and using the triangle inequality, we get for any $t \le T$,
	\begin{align}
		\nnorm{Z_t^i} &= \nnnorm{X_0^{i} + \int_0^t b(s,\,X_s^i,\nu_s^N,\mathfrak{f}_s^N)\diff{s} + Y_t^{i}}\notag\\
		&\le \nnnorm{X_0^{i}} + c_1\int_0^t 1 + \sup_{u \le s} \nnnorm{Z_u^i} + \frac{1}{N} \sum_{j =1}^N \sup_{u \le s} \nnnorm{Z_u^j}\diff{}s + \nnorm{Y_t^{i}}\notag\\
		&\le \nnorm{X_0^{i}} + tc_1+  c_1\int_0^t \sup_{u \le s} \nnorm{Z_u^i} + \frac{1}{N}\sum_{j = 1}^N \sup_{u \le s} \nnorm{Z_u^j} \diff{}s + \sup_{s\le t}\nnnorm{Y_s^i},\label{eq: FIRST EQUATION IN THE PROOF THAT A UNIQUE EQUATION EXISTS TO THE SYSTEM OF MOLLIFIED AND SMOOTHED DIFFUSIONS WITHOUT RESETS}
	\end{align}
	where $Y_t^i = \int_0^t \sigma(s,X_s^i)\ind_{\{X_s^i < 0\}}\diff{}B_s^i$ with $\diff{}B_s^i = \sqrt{1 - \rho^2_s}\diff{} W_s^i+\rho_s\diff{}W_s^0$ and $\rho_s = \rho(s,\nu^N_s, \mathfrak{f}_s^N)$. By taking the empirical mean above, we deduce
	\begin{equation*}
		\frac{1}{N}\sum_{j = 1}^N \sup_{s \le t} \nnorm{Z_s^j} \le tc_1 + \frac{1}{N}\sum_{j = 1}^N \left(\nnorm{X_0^j} + \sup_{s \le t} \nnorm{Y_s^{j}}\right) + 2c_1\int_0^t \frac{1}{N}\sum_{j = 1}^N \sup_{u \le s} \nnorm{Z_u^j} \diff{}s.
	\end{equation*}
	Then, by Gr\"onwall's integral lemma, we obtain
	\begin{equation}\label{eq: SECOND EQUATION IN THE PROOF THAT A UNIQUE EQUATION EXISTS TO THE SYSTEM OF MOLLIFIED AND SMOOTHED DIFFUSIONS WITHOUT RESETS}
		\frac{1}{N}\sum_{j = 1}^N \sup_{s \le t} \nnorm{Z_s^j} \le C\left[1 + \frac{1}{N}\sum_{j = 1}^N \left(\nnorm{X_0^j} + \sup_{s \le t} \nnorm{Y_s^{j}}\right)\right] \quad \text{a.s.}
	\end{equation}
	for some possibly larger constant $C > 0$ that depends on $T$, $b$, and $\mathfrak{K}$ only. Now, going back to \eqref{eq: FIRST EQUATION IN THE PROOF THAT A UNIQUE EQUATION EXISTS TO THE SYSTEM OF MOLLIFIED AND SMOOTHED DIFFUSIONS WITHOUT RESETS} and substituting \eqref{eq: SECOND EQUATION IN THE PROOF THAT A UNIQUE EQUATION EXISTS TO THE SYSTEM OF MOLLIFIED AND SMOOTHED DIFFUSIONS WITHOUT RESETS} into it, we obtain
	\begin{equation*}
		\nnnorm{Z_t^i} \le \nnorm{X_0^{i}} + \sup_{s\le t}\nnnorm{Y_s^i} + C\left[1 + \frac{1}{N}\sum_{j = 1}^N \left(\nnorm{X_0^j} + \sup_{s \le t} \nnorm{Y_s^{j}}\right)\right] +  c_1\int_0^t \sup_{u \le s} \nnorm{Z_u^i}\diff{}s \quad \text{a.s.}
	\end{equation*}
	for some possibly larger constant $C > 0$ that still depends only on $T$, $b$, and $\mathfrak{K}$. By using Gr\"onwall's lemma again, there is a constant $C_{T,b,\mathfrak{K}} > 0$ such that for any $t \le T$,
	\begin{equation*}
		\sup_{s \le t}\nnorm{Z_s^i} \le C_{T,b,\mathfrak{K}}\left[1 + \nnnorm{X_0^i} + \sup_{s \le t} \nnnorm{Y_s^i} +  \frac{1}{N}\sum_{j = 1}^N \left(\nnnorm{X_0^j} + \sup_{s \le t} \nnnorm{Y_s^j}\right)\right]\quad \text{a.s.}
	\end{equation*}
	Jensen's inequality and the equation above gives, for any $p \ge 1$,
	\begin{equation}\label{eq: UPPERBOUND ON THE SUP OF ZS FOR ANY I}
		\sup_{s \le t}\nnorm{Z_s^i}^p \le C_{T,b,\mathfrak{K},p}\left[1 + \nnnorm{X_0^i}^p + \sup_{s \le t} \nnnorm{Y_s^i}^p +  \frac{1}{N}\sum_{j = 1}^N \left(\nnnorm{X_0^j}^p + \sup_{s \le t} \nnnorm{Y_s^j}^p\right)\right]\quad \text{a.s.}
	\end{equation}
	By Burkholder-Davis-Gundy, for any $j$ 
	\begin{equation*}
		\E\left[\sup_{0\le s \le t} \nnnorm{\int_0^s \sigma(u,X_u^j)\ind_{\{X_u^j < 0\}}\diff{}B_u^j}^p\right] \le \E\left[\left(\int_0^T \sigma^2(s,X_s^j)\diff{}s\right)^{p/2}\right] \le C_{\sigma,p} T^{p/2}.
	\end{equation*}
	Hence,
	\begin{equation}\label{eq: LAST EQUATION IN THE PROOF THAT THE SUP OF Z HAS FINITE P MOMENTS}
		\E\left[\sup_{0\le s \le t}\nnnorm{Z_s^i}^p\right] \le
		C_{T,b,\mathfrak{K},p}\left[1 + \E\nnnorm{X_0^i}^p + 2C_{\sigma,p} T^{p/2} +  \frac{1}{N}\sum_{j = 1}^N \E\nnnorm{X_0^j}^p\right].
	\end{equation}
	As $X_0$ has finite $p^{\text{th}}$-moments for every $p > 0$ by \cref{ass: ASSUMPTIONS FROM SOJMARK SPDE PAPER APPLIED TO THE NEUROSCIENCE SETTING} \eqref{ass: ASSUMPTIONS FROM SOJMARK SPDE PAPER APPLIED TO THE NEUROSCIENCE SETTING FOUR}, the R.H.S. of \eqref{eq: LAST EQUATION IN THE PROOF THAT THE SUP OF Z HAS FINITE P MOMENTS} is bounded by a constant that depends on $p,\,b,\,\mathfrak{K},\,\sigma$, and $T$ but is independent of $N$.
\end{proof}

From the above, it is clear that the following corollary holds.
\begin{corollary}[Uniform moment bounds]\label{prop: BOUNDS ON THE SUP OF THE P TH MOMENTS OF THE SUP NORM OF X}
	For any $p \ge 1$, we have 
	\begin{equation*}
		\E\left[\sup_{0\le t \le T}\nnorm{X_t^i}^p + \left(J_T^i\right)^p + \left(F_T^i\right)^p + \left(F_T^N\right)^p\right] \le C_{T,\,p},
	\end{equation*}
	where $C_{T,\,p} > 0$ is independent of $i$ and $N$.
\end{corollary}

Our aim is to show that $X_t^{i},\,J_t^{i},$ and $F_t^N$ are all sub-Gaussian uniformly in $N$. To this end, it will be sufficient to show that the processes
\begin{align}
	&\Lambda_t^i \coloneqq \nnnorm{Z_t^i} + \frac{1}{N}\sum_{j = 1}^N \nnnorm{Z_t^j}, &\Gamma_t^i \coloneqq \nnnorm{Z_t^{i}}^2 + \frac{1}{N}\sum_{j = 1}^N \nnnorm{Z_t^j}^2,\label{eq: DEFINITION OF LAMBDA AND GAMMA}\\
	&\hat{\Lambda}_t^i \coloneqq \sup_{0 \le s \le t} \nnnorm{Z_s^i} + \frac{1}{N}\sum_{j = 1}^N \sup_{0 \le s \le t} \nnnorm{Z_s^j},  &\hat{\Gamma}_t^{i} \coloneqq \sup_{0 \le s \le t}\nnnorm{Z_s^i}^2 + \frac{1}{N}\sum_{j = 1}^N \sup_{0 \le s \le t}\nnnorm{Z_s^j}^2.\label{eq: DEFINITION OF LAMBDA HAT AND GAMMA HAT}
\end{align}
all have finite exponential moments. We introduce this notation here and consider these random variables, as they play a role in Girsanov-type arguments presented later. These show that we have a probability space where we can view our particle as being driven by only a Brownian motion. This notation is exactly that presented in \cite[Section~6.1]{hambly2019spde}.

\begin{proposition}\label{prop: UNIFORM BOUNDS ON THE EXPONENTIAL OF GAMMA HAT T}
	For $\delta < \min\{\gamma/4, 1/8C_\sigma^2 T\}$, where $\gamma$ is the constant in \cref{ass: ASSUMPTIONS FROM SOJMARK SPDE PAPER APPLIED TO THE NEUROSCIENCE SETTING} \eqref{ass: ASSUMPTIONS FROM SOJMARK SPDE PAPER APPLIED TO THE NEUROSCIENCE SETTING FOUR}, and any $i = 1,\ldots, N$, there exists a constant $C > 0$ such that
	\begin{equation*}
		\E\left[e^{\delta \hat{\Gamma}_T^i}\right] \le C.
	\end{equation*}
	The constant $C$ depends on $\delta,\,T,\,\mathfrak{K}$, the initial distribution $X_0$, and the bounds on $b$ and $\sigma$, but is independent of $i$ and $N$.
\end{proposition}
\begin{proof}
	By \cref{prop: BOUNDS ON THE SUP OF THE P TH MOMENTS OF THE SUP NORM OF Z}, for any $i = 1,\ldots, N$ fixed, we have by \eqref{eq: UPPERBOUND ON THE SUP OF ZS FOR ANY I},
	\begin{equation*}
		\sup_{s \le T}\nnorm{Z_s^i}^2 \le C_{T,b,\mathfrak{K}}\left[1 + \nnnorm{X_0^i}^2 + \sup_{s \le T} \nnnorm{Y_s^i}^2 +  \frac{1}{N}\sum_{j = 1}^N \left(\nnnorm{X_0^j}^2 + \sup_{s \le T} \nnnorm{Y_s^j}^2\right)\right]\quad \text{a.s.},
	\end{equation*}
	where $C_{T,b,\mathfrak{K}}$ is a constant that depends on $T,\,b$, and $\mathfrak{K}$ only, and $Y_t^i = \int_0^t \sigma(s,X_s^i)\ind_{\{X_s^i < 0\}}\diff{}B_s^i$ with $\diff{}B_s^i = \sqrt{1 - \rho^2_s}\diff{} W_s^i+\rho_s\diff{}W_s^0$ and $\rho_s = \rho(s,\nu^N_s,\mathfrak{f}_s^N)$. By the Generalised H\"older inequality, we have
	\begin{align*}
		&\E\left[e^{\delta \hat{\Gamma}_t^i}\right] \\
		&\le e^{C_{T,b,\mathfrak{K}}} \E \left[e^{4\delta \nnnorm{X_0^i}^2}\right]^{\frac{1}{4}} \E \left[e^{4\delta \sup_{s \le T} \nnnorm{Y_s^i}^2}\right]^{\frac{1}{4}} \E \left[e^{ \frac{4\delta}{N}\sum_{j = 1}^N \nnnorm{X_0^j}^2}\right]^{\frac{1}{4}} \E \left[e^{ \frac{4\delta}{N}\sum_{j = 1}^N \sup_{s \le T} \nnnorm{Y_s^j}^2}\right]^{\frac{1}{4}}.
	\end{align*}
	The goal is now to upper bound each term on the R.H.S. above by a constant that is independent of $N$ and $i$. As $\delta < \gamma/4$ by assumption, then by \cref{ass: ASSUMPTIONS FROM SOJMARK SPDE PAPER APPLIED TO THE NEUROSCIENCE SETTING} \eqref{ass: ASSUMPTIONS FROM SOJMARK SPDE PAPER APPLIED TO THE NEUROSCIENCE SETTING FOUR}
	\begin{equation*}
		\E \left[e^{4\delta \nnnorm{X_0^j}^2}\right] \le C_{2,4\delta,X_0} \qquad \forall \, j\in \{1,\ldots,N\}.
	\end{equation*}
	Therefore, by Generalised H\"older inequality and the above, we have
	\begin{equation*}
		\E \left[e^{ \frac{4\delta}{N}\sum_{j = 1}^N \nnnorm{X_0^j}^2}\right] \le \prod_{j = 1}^N \E \left[e^{ {4\delta}\nnnorm{X_0^j}^2}\right]^{\frac{1}{N}} \le C_{2,4\delta,X_0}.
	\end{equation*}
	For any $j$, as $Y_t^j$ is a continuous local martingale, by the Dubins-Schwarz Theorem there is a Brownian motion $\hat{B}$ such that $Y_t^j = \hat{B}_{\int_0^t \sigma^2(s,X_s^j)\ind_{\{X_s^j < 0\}} \diff s}$. Therefore,
	\begin{align*}
		\E \left[e^{4\delta \sup_{t \le T} \nnnorm{Y_t^j}^2}\right] 
		&= \E \left[e^{4\delta \sup_{t \le T} \nnnorm{\hat{B}_{\int_0^t \sigma^2(s,X_s^j)\ind_{\{X_s^j < 0\}} \diff s}}^2}\right] \\ 
		&\le \E \left[e^{4\delta \sup_{t \le T} \nnnorm{\hat{B}_{C_\sigma^2t}}^2}\right] \le C_{p,4\delta,C_\sigma^2T},
	\end{align*}
	where the penultimate line follows from the bounds on $\sigma$, \cref{ass: ASSUMPTIONS FROM SOJMARK SPDE PAPER APPLIED TO THE NEUROSCIENCE SETTING} \eqref{ass: ASSUMPTIONS FROM SOJMARK SPDE PAPER APPLIED TO THE NEUROSCIENCE SETTING ONE}, and the last line follows from 
	well known results on the exponential moments of Brownian motion.
	By the Generalised H\"older inequality and the above, we have
	\begin{equation*}
		\E\left[e^{ \frac{4\delta}{N}\sum_{j = 1}^N \sup_{s \le T} \nnnorm{Y_s^j}^2}\right] \le \prod_{j = 1}^N \E \left[e^{ {4\delta}\sup_{s \le T} \nnnorm{Y_s^j}^2}\right]^{\frac{1}{N}} \le C_{p,4\delta,C_\sigma^2T}.
	\end{equation*}
	Therefore, we have shown that
	\begin{equation*}
		\E\left[e^{\delta \hat{\Gamma}_t^i}\right] \le e^{C_{T,b,\mathfrak{K}}} \cdot (C_{2,4\delta,X_0})^{1/2}\cdot (C_{p,4\delta,C_\sigma^2T})^{1/2}.
	\end{equation*}
\end{proof}

\begin{corollary}\label{cor: SUBGUASSIANTY COROLLARY}
	$X_t^{i},\,Z_t^{i},\,J_t^{i},\,F_t^N$, and $\hat{\Lambda}_t^{i}$ are all sub-Gaussian uniformly in $N \ge 1$ and $t \in [0,T]$.
\end{corollary}
\begin{proof}
	The claim follows directly from \cref{prop: UNIFORM BOUNDS ON THE EXPONENTIAL OF GAMMA HAT T}.
\end{proof}

Similarly, as in \cite[Section~6.2]{hambly2019spde}, we may transform each particle into a Brownian motion with drift, which pauses for a random amount of time when the particle hits $0$. Then, the particle is reset to a random value below $0$. Additionally, this transformation is such that the times the transformed particle hits $0$ are equal to the times the original particle hits $0$. This transformation is crucial for the Girsanov-type arguments that follow.

\begin{proposition}\label{prop: SCALE TRANSFORM LEMMA }
	Define the transformation $\Upsilon \in \mathcal{C}^{1,\,2}\left([0,\,T]\times\R\right)$ by
	\begin{equation*}
		(t,\,x) \mapsto \Upsilon_t(x) \coloneqq \int_0^x\frac{\diff{y}}{\sigma(t,y)}.
	\end{equation*}
	Fixing an arbitrary index $i \in \{1,\,\ldots,\,N\}$, then 
	\begin{equation*}
		\diff\Upsilon_t\left(X_t^i\right) = \hat{b}_t^i \ind_{\{\Upsilon_t\left(X_t^i\right) < 0\}}\diff{t} + \ind_{\{\Upsilon_t\left(X_t^i\right) < 0\}} \diff{B_t^i} + \diff\sum_{k \ge 1} \Upsilon_{\tau_k^i + \varsigma_k^i}(-\xi_k^i)\ind_{[0,t]}(\tau_k^i + \varsigma_k^i),
	\end{equation*}
	where $B^i$ is a Brownian motion and the (stochastic) drift $\hat{b}_t^{i}$ satisfies the growth condition
	\begin{equation}\label{eq: GROWTH CONDITION ON THE EQUATION OF B HAT IN CHANGE OF MEASURE}
		\nnorm{\hat{b}_t^{i}} \le C_{b,\sigma,\mathfrak{K}} \left(1 + \nnnorm{X_t^i} + \langle\nu_t^N,\nnnorm{\cdot}\rangle + \nnnorm{{F}_t^{N}}\right).
	\end{equation}
	Furthermore, the transformed process $\Upsilon_t\left(X_t^i\right)$ satisfies
	\begin{equation}\label{eq: BOUND ON MATHFRAK Z WHEN CHANGING THE MEASURE}
		\operatorname{sgn}\left(\Upsilon_t\left(X_t^i\right)\right) = \operatorname{sgn}\left(X_t^i\right) \quad \text{and} \quad \nnnorm{\Upsilon_t\left(X_t^i\right)} \le C\nnnorm{X_t^i}.
	\end{equation}
\end{proposition}

\begin{proof}
	As before, we define $\diff{B_t^i} \coloneqq \rho(t,\nu_t^N,\mathfrak{f}_t^N)\diff{W_t^0} + \sqrt{1 - \rho^2(t,\nu_t^N,\mathfrak{f}_t^N)}\diff{W_t^i}$. Therefore, by applying It\^o's formula with jumps \citep[Chapter~II,Theorem~33]{protter2005stochastic}, we have
	\begin{align}\label{eq: GENERALISED ITO FORMULA FOR NEUORN IN DELAYED JUMPS IN SCALE TRANSFORM PROPOSITION}
		\begin{split}
			\Upsilon_t(X_t^i) =& \Upsilon_0(X_0^i) + \int_0^t \partial_s \Upsilon_s(X_s^i)\diff s + \int_{0+}^t \partial_x \Upsilon_{s-}(X_{s-}^i)\diff{}X_s^i \\
			&+ \frac{1}{2}\int_{0+}^t\partial_x^2\Upsilon_{s-}(X_{s-}^i)\diff{}[X^i]^c_s \\
			&+ \sum_{\substack{0 < s \le t \\ s \text{ is jump time}}} \{\Upsilon_s(X_s^i) - \Upsilon_{s-}(X_{s-}^i) - \partial_x\Upsilon_{s}(X_{s-}^i)\Delta X_s^i\}.
		\end{split}
	\end{align}
	Since $X^i$ has \cadlag paths and will jump to $-\xi_k^i$ only at times $\tau_k^i + \varsigma_k^i$, and $X^i$, by construction, is $0$ at time $(\tau_k^i + \varsigma_k^i)-$, and also noting that $\Upsilon_s(0) = 0$ for every $s$, \eqref{eq: GENERALISED ITO FORMULA FOR NEUORN IN DELAYED JUMPS IN SCALE TRANSFORM PROPOSITION} expands to
	\begin{align*}
		\begin{split}
			\Upsilon_t(X_t^i) &= \Upsilon_0(X_0^i) + \int_0^t \partial_s \Upsilon_s(X_s^i)\diff s +  \int_{0}^t \partial_x \Upsilon_{s}(X_{s}^i)
			b(s,\,X_s^i,\nu_s^{N},\mathfrak{f}_s^N)\ind_{\{X_s^i < 0\}}\diff{}s\\
			&+ \int_{0}^t\partial_x\Upsilon_{s}(X_{s}^i)\sigma(s,\,X_s^i)\ind_{\{X_s^i < 0\}}\diff{}B_s^i -\int_{0+}^t\partial_x \Upsilon_{s-}(X_{s-}^i)\diff{}\left(\sum_{k \ge 1} \xi_k^i \ind_{[0,s]}(\tau_k^i + \varsigma_k^i)\right)\\
			&+ \frac{1}{2}\int_{0}^t\partial_x^2\Upsilon_{s}(X_{s}^i)\sigma^2(s,\,X_s^i)\ind_{\{X_s^i < 0\}}\diff{}s + \sum_{\substack{0 < s \le t \\ s \text{ is jump time}}} \{\Upsilon_s(X_s^i) - \partial_x\Upsilon_{s}(X_{s-}^i)\Delta X_s^i\}.
		\end{split}
	\end{align*}
	As $\sum_{k \ge 1} \xi_k^i \ind_{[0,s]}(\tau_k^i + \varsigma_k^i)$ is a pure jump process, we observe that
	\begin{equation*}
		-\int_{0+}^t\partial_x \Upsilon_{s-}(X_{s-}^i)\diff{}\left(\sum_{k \ge 1} \xi_k^i \ind_{[0,s]}(\tau_k^i + \varsigma_k^i)\right) = \sum_{\substack{0 < s \le t \\ s \text{ is jump time}}} \partial_x\Upsilon_{s}(X_{s-}^i)\Delta X_s^i.
	\end{equation*}
	Therefore,
	\begin{align}
		\Upsilon_t(X_t^i) &= \Upsilon_0(X_0^i) + \int_0^t \partial_s \Upsilon_s(X_s^i)\diff s +  \int_{0}^t \partial_x \Upsilon_{s}(X_{s}^i)
		b(s,\,X_s^i,\nu_s^{N},\mathfrak{f}_s^N)\ind_{\{X_s^i < 0\}}\diff{}s \notag \\
		&+ \int_{0}^t\partial_x\Upsilon_{s}(X_{s}^i)\sigma(s,\,X_s^i)\ind_{\{X_s^i < 0\}}\diff{}B_s^i + \frac{1}{2}\int_{0}^t\partial_x^2\Upsilon_{s}(X_{s}^i)\sigma^2(s,\,X_s^i)\ind_{\{X_s^i < 0\}}\diff{}s \notag \\
		& + \sum_{\substack{0 < s \le t \\ s \text{ is jump time}}} \Upsilon_s(X_s^i). \label{eq: GENERALISED ITO FORMULA FOR NEUORN IN DELAYED JUMPS IN SCALE TRANSFORM PROPOSITION EXPANDED}
	\end{align}
	A straightforward computation shows that
	\begin{equation*}
		\partial_t \Upsilon_t(x) = -\int_0^x\frac{\partial_t \sigma(t,\,y)}{\sigma(t,\,y)^2}\diff{y},
		\qquad
		\partial_x \Upsilon_t(x) = \frac{1}{\sigma(t,\,x)},
		\qquad
		\partial_x^2\Upsilon_t(x) = -\frac{\partial_x \sigma(t,\,x)}{\sigma(t,\,x)^2}.
	\end{equation*}
	We observe that $\partial_t \Upsilon_t(0) = 0$; therefore, it follows from \eqref{eq: GENERALISED ITO FORMULA FOR NEUORN IN DELAYED JUMPS IN SCALE TRANSFORM PROPOSITION EXPANDED} that
	\begin{equation*}
		\diff\Upsilon_t\left(X_t^i\right) = \hat{b}_t^i \ind_{\{X_t^i< 0\}}\diff{t} + \ind_{\{X_t^i < 0\}} \diff{B_t^i} + \diff\sum_{k \ge 1} \Upsilon_{\tau_k^i + \varsigma_k^i}(-\xi_k^i)\ind_{[0,t]}(\tau_k^i + \varsigma_k^i),
	\end{equation*}
	with
	\begin{equation*}
		\hat{b}_t^i \coloneqq \frac{b(t,X_t^i,\nu_t^N,\mathfrak{f}_t^N)}{\sigma(t,X_t^i)} - \frac{1}{2}\partial_x\sigma(t,X_t^i) - \int_0^{X_t^i} \frac{\partial_t\sigma(t,y)}{\sigma(t,y)^2}\diff{y}.
	\end{equation*}
	Now, the bound on $\Upsilon_t\left(X_t^i\right)$ and the statement about its sign in \eqref{eq: BOUND ON MATHFRAK Z WHEN CHANGING THE MEASURE} follow directly from the definition of $\Upsilon$, since $\sigma$ is strictly positive and bounded away from zero. Similarly, the growth condition in \eqref{eq: GROWTH CONDITION ON THE EQUATION OF B HAT IN CHANGE OF MEASURE} follows from the linear growth of $b$ along with $|\mathfrak{f}_t^N| \le \norm{\mathfrak{K}^\prime}_{L^1}F_t^N$.
\end{proof}

\subsection{Change of Measure}

Based on the scale transformation \cref{prop: SCALE TRANSFORM LEMMA }, we can exploit sub-Gaussianity to introduce a change of measure that removes the drift of a given particle. This will be convenient for estimates on the behaviour of the empirical measures in expectation.

\begin{proposition}\label{prop: NOVIKOV CONDITION FOR THE STOCHASTIC EXPONENTIAL}
	Fix $i \in \{1,\ldots,\,N\}$ and define the stochastic exponential to be
	\begin{equation*}
		\mathcal{E}_t \coloneqq \exp\left\{-\int_0^t \hat{b}_s^i \ind_{\{\Upsilon_s\left(X_s^i\right) < 0\}} \diff{}B_s^i - \frac{1}{2} \int_0^t \left(\hat{b}_s^i\right)^2 \ind_{\{\Upsilon_s\left(X_s^i\right) < 0\}} \diff{}s\right\},
	\end{equation*}
	where $\hat{b}_t^i \ind_{\{\Upsilon_t\left(X_t^i\right) < 0\}}$ is the drift of $\Upsilon_t\left(X_t^i\right)$ defined as in \cref{prop: SCALE TRANSFORM LEMMA }. Then $\int_0^t\hat{b}_s^i \ind_{\{\Upsilon_s\left(X_s^i\right) < 0\}}\diff{s} + {B_t^i}$ is a standard Brownian motion under the probability measure $\mathbbm{Q}$ given by the Radon-Nikodym derivative
	\begin{equation}
		\left.\frac{\diff{}\mathbbm{Q}}{\diff{}\prob}\right|_{\mathcal{F}_t} = \mathcal{E}_t.
	\end{equation}
\end{proposition}

\begin{proof}
	The proof follows exactly that of \citep[Lemma~6.4]{hambly2019spde} with the only change being $\nnnorm{\hat{b}_s^i \ind_{\{\Upsilon_s\left(X_s^i\right) < 0\}}} \le C(1 + \hat{\Lambda}_s^i)$.
\end{proof}

In addition to the change of measure, we will also need an estimate on the expected value of $\mathcal{E}_t^{1-p}$ for $p > 1$ close to 1, so that we may employ H\"older-type estimates to bound the probabilities of $X^i$ for any $i$.

\begin{proposition}[Radon-Nikodym Estimate]\label{prop: BOUNDS ON THE 1 - P NORM OF THE RADON NIKODYM ESTIMATE}
	Let $\mathcal{E}_t$ be the stochastic exponential defined as in \cref{prop: NOVIKOV CONDITION FOR THE STOCHASTIC EXPONENTIAL}. Then, for all $p > 1$ close enough to $1$,
	\begin{equation*}
		\E\left[\mathcal{E}_t^{1-p}\right] \le C,
	\end{equation*}
	where $C > 0$ is a constant that depends on $p,\,T,\,\mathfrak{K}$, the initial distribution $X_0$, and the bounds on $b$ and $\sigma$, but is independent of $i$ and $N$. Moreover, for any $q > 1$, we have
	\begin{equation*}
		\E\left[\nnorm{\hat{b}_t^{i}}^q \right] \le \tilde{C},
	\end{equation*}
	where $\tilde{C} > 0$ is a constant that depends on $q,\,T,\,\mathfrak{K}$, the initial distribution $X_0$, and the bounds on $b$ and $\sigma$, but is independent of $i$ and $N$.
\end{proposition}

\begin{proof}
	We define $\diff{}Y_t =\hat{b}_t^i\ind_{\{\Upsilon_t\left(X_t^i\right) < 0\}}\diff{}B_t^i$. Then, for any $p > 1$, letting $\mathcal{E}(pY)$ denote the stochastic exponential of $pY$, we observe
	
	\begin{align*}
		\E\left[\mathcal{E}_t^{1 - p}\right] &=
		\E\left[\exp\left\{-(1-p)Y_t - \frac{1-p}{2}\langle Y \rangle_t\right\}\right]\\
		&= \E\left[\exp\left\{pY_t - \frac{1}{2}\langle pY\rangle_t\right\}^{\frac{p-1}{p}}\exp\left\{\frac{p(p+1)(p-1)}{2}\langle Y\rangle_t\right\}^{\frac{1}{p}}\right]\\
		&\le \E\left[\mathcal{E}(pY)_t\right]^\frac{p-1}{p}\E\left[\exp\left\{\frac{p(p+1)(p-1)}{2}\langle Y\rangle_t\right\}\right]^{\frac{1}{p}},
	\end{align*}
	where the last line follows from H\"older's inequality. Using the same proof as in \cref{prop: NOVIKOV CONDITION FOR THE STOCHASTIC EXPONENTIAL}, we have that $\mathcal{E}(pY)$ is a martingale. Therefore, the first term is bounded by $1$. Hence, we deduce
	\begin{equation*}
		\E\left[\mathcal{E}_t^{1 - p}\right]  \le \E\left[\exp\left\{C_p \int_0^t (\hat{b}_s^i)^2\ind_{\{\Upsilon_s\left(X_s^i\right) < 0\}}\diff{}s\right\}\right]^{\frac{1}{p}},
	\end{equation*}
	with $C_p = p(p-1)(p+1)/2$. Since $\nnorm{\hat{b}_s^i} \le C_{b,\sigma,\mathfrak{K}}(1 + \hat{\Lambda}_t^i)$ by \eqref{eq: GROWTH CONDITION ON THE EQUATION OF B HAT IN CHANGE OF MEASURE}, we have
	\begin{equation*}
		\E\left[\mathcal{E}_t^{1 - p}\right]  \le e^{\frac{2TC_{b,\sigma,\mathfrak{K}}C_p}{p}}\E\left[\exp\left\{TC_{b,\sigma,\mathfrak{K}}C_p (\hat{\Lambda}_t^{i})^2\right\}\right]^{\frac{1}{p}}.
	\end{equation*}
	Now, by choosing $p \in (1,p_0)$, where $p_0 \coloneqq \sup \{p > 1\,:\, 2TC_pC < \min\{\gamma/4, 1/8C_\sigma^2 T\}\}$ and $ \gamma$ is the constant in \cref{ass: ASSUMPTIONS FROM SOJMARK SPDE PAPER APPLIED TO THE NEUROSCIENCE SETTING} \eqref{ass: ASSUMPTIONS FROM SOJMARK SPDE PAPER APPLIED TO THE NEUROSCIENCE SETTING FOUR}, we have
	\begin{align*}
		\E\left[\mathcal{E}_t^{1 - p}\right]
		&\le e^{\frac{2TC_{b,\sigma,\mathfrak{K}}C_p}{p}}\E\left[\exp\left\{TC_pC(\hat{\Lambda}_t^{i})^2\right\}\right]^{\frac{1}{p}}\\
		&\le e^{\frac{2TC_{b,\sigma,\mathfrak{K}}C_p}{p}}\E\left[\exp\left\{2TC_pC \hat{\Gamma}_t^{i} \right\}\right]^{\frac{1}{p}} \leq C.
	\end{align*}
	The last line follows from \cref{prop: UNIFORM BOUNDS ON THE EXPONENTIAL OF GAMMA HAT T}, and the constant depends on $p,\,T,\,\mathfrak{K}$, the initial distribution $X_0$, and the bounds on $b$ and $\sigma$, but is independent of $i$ and $N$.
	
	For the second claim, by \cref{prop: SCALE TRANSFORM LEMMA } we deduce that $\nnorm{\hat{b}_s^i} \le C(1 + \hat{\Lambda}_s^i)$. Hence, the claim follows from applying \cref{cor: SUBGUASSIANTY COROLLARY}.
\end{proof}

\subsection{Decay of Hitting Time Probabilities}
With the results from the previous subsections, we can finally establish an exponential decay property of the hitting time probabilities.

\begin{proposition}[Exponential decay of hitting times]\label{prop: EXPONENTIAL DECAY OF THE HITTING PROBABILITIES}
	For any $i,k \ge 1$ and for all $p > 1$, sufficiently close to $1$, we have
	\begin{equation*}
		\prob\left[\tau_k^i \le T\right] \le C^{\frac{1}{p}}e^{-\frac{(k-1)^2\gamma^2(p-1)}{2pTC_\sigma^2}},
	\end{equation*}
	for some constant $C > 0$ that depends on $p$ and $T$ but holds uniformly in $i,\,k$, and $N$, and $ \gamma$ is the constant in \cref{ass: ASSUMPTIONS FROM SOJMARK SPDE PAPER APPLIED TO THE NEUROSCIENCE SETTING} \eqref{ass: ASSUMPTIONS FROM SOJMARK SPDE PAPER APPLIED TO THE NEUROSCIENCE SETTING FOUR}.
\end{proposition}

\begin{proof}[Proof of \cref{prop: EXPONENTIAL DECAY OF THE HITTING PROBABILITIES}]
	To begin, we fix $i,\,k \ge 1$. By the definition of $\Upsilon_t(x)$, the $k^{\textnormal{th}}$ hitting time of $X^i$ is the $k^{\textnormal{th}}$ hitting time of $\Upsilon_t(X_t^i)$. As
	\begin{equation*}
		\diff\Upsilon_t\left(X_t^i\right) = \hat{b}_t^i \ind_{\{\Upsilon_t\left(X_t^i\right) < 0\}}\diff{t} + \ind_{\{\Upsilon_t\left(X_t^i\right) < 0\}} \diff{B_t^i} + \diff\sum_{\tilde{k} \ge 1} \Upsilon_{\tau_{\tilde{k}}^i + \varsigma_{\tilde{k}}^i}(-\xi_{\tilde{k}}^i)\ind_{[0,t]}(\tau_{\tilde{k}}^i + \varsigma_{\tilde{k}}^i)
	\end{equation*}
	by \cref{prop: SCALE TRANSFORM LEMMA }, then $\tau_k^i$ corresponds to the first time the process $\Upsilon_0(X_0^i) +  \int_0^t \ind_{\{\Upsilon_s\left(X_s^i\right) < 0\}} \diff{}\tilde{B}_s^i = \Upsilon_t\left(X_t^i\right) - \sum_{{\tilde{k}} \ge 1} \Upsilon_{\tau_{\tilde{k}}^i + \varsigma_{\tilde{k}}^i}(-\xi_{\tilde{k}}^i)\ind_{[0,t]}(\tau_{\tilde{k}}^i + \varsigma_{\tilde{k}}^i)$ hits the level $-\sum_{{\tilde{k}} = 1}^{k-1} \Upsilon_{\tau_{\tilde{k}}^i + \varsigma_{\tilde{k}}^i}(-\xi_{\tilde{k}}^i)\ind_{[0,t]}(\tau_{\tilde{k}}^i + \varsigma_{\tilde{k}}^i)$. Here, $\tilde{B}_t^i = \int_0^t\hat{b}_s^i \ind_{\{\Upsilon_s\left(X_s^i\right) < 0\}}\diff{s} + {B_t^i}$ is a $\mathbbm{Q}$-standard Brownian motion by \cref{prop: NOVIKOV CONDITION FOR THE STOCHASTIC EXPONENTIAL}. Thus,
	{\small\begin{equation*}
			\prob\left[\tau_k^i \le T\right] = \prob\left[\Upsilon_0(X_0^i) + \sup_{t \le T} \int_0^t \ind_{\{\Upsilon_s\left(X_s^i\right) < 0\}} \diff{}\tilde{B}_s^i \ge -\sum_{{\tilde{k}} = 1}^{k-1} \Upsilon_{\tau_{\tilde{k}}^i + \varsigma_{\tilde{k}}^i}(-\xi_{\tilde{k}}^i)\ind_{[0,t]}(\tau_{\tilde{k}}^i + \varsigma_{\tilde{k}}^i) \right]. 
	\end{equation*}}
	By the bounds on $\sigma$ and $\xi_{\tilde{k}}^i$ given in \cref{ass: ASSUMPTIONS FROM SOJMARK SPDE PAPER APPLIED TO THE NEUROSCIENCE SETTING} \eqref{ass: ASSUMPTIONS FROM SOJMARK SPDE PAPER APPLIED TO THE NEUROSCIENCE SETTING ONE} and \eqref{ass: ASSUMPTIONS FROM SOJMARK SPDE PAPER APPLIED TO THE NEUROSCIENCE SETTING FOUR}, we have
	\begin{equation*}
		- \Upsilon_{\tau_{\tilde{k}}^i + \varsigma_{\tilde{k}}^i}(-\xi_{\tilde{k}}^i) = \int_{-\xi_{\tilde{k}}^i}^{0} \frac{\diff y}{\sigma(\tau_{\tilde{k}}^i + \varsigma_{\tilde{k}}^i,y)} \ge \xi_{\tilde{k}}^iC_\sigma^{-1} \ge \gamma C_{\sigma}^{-1}.
	\end{equation*}
	Therefore,
	\begin{align*}
		\prob\left[\tau_k^i \le T\right] &= \prob\left[\Upsilon_0(X_0^i) + \sup_{t \le T} \int_0^t \ind_{\{\Upsilon_s\left(X_s^i\right) < 0\}} \diff{}\tilde{B}_s^i \ge -\sum_{{\tilde{k}} = 1}^{k-1} \Upsilon_{\tau_{\tilde{k}}^i + \varsigma_{\tilde{k}}^i}(-\xi_{\tilde{k}}^i)\ind_{[0,t]}(\tau_{\tilde{k}}^i + \varsigma_{\tilde{k}}^i) \right]\\
		&\le \prob\left[\Upsilon_0(X_0^i) + \sup_{t \le T} \int_0^t \ind_{\{\Upsilon_s\left(X_s^i\right) < 0\}} \diff{}\tilde{B}_s^i \ge (k-1)\gamma C_\sigma^{-1}\right] \\
		&\le \prob\left[\sup_{t \le T} \int_0^t \ind_{\{\Upsilon_s\left(X_s^i\right) < 0\}} \diff{}\tilde{B}_s^i \ge (k-1)\gamma C_\sigma^{-1}\right], \\ \intertext{where the last line follows from the fact that $\Upsilon_0(X_0^i) < 0$ a.s. because $X_0^i < 0$ a.s. Then, for all $p > 1$ and close enough to $1$, by H\"older's inequality,}
		&\le \E\left[\mathcal{E}_t^{1 - p}\right]^{\frac{1}{p}}\mathbbm{Q}\left[\sup_{t \le T} \int_0^t \ind_{\{\Upsilon_s\left(X_s^i\right) < 0\}} \diff{}\tilde{B}_s^i \ge (k-1)\gamma C_\sigma^{-1}\right]^{\frac{p-1}{p}}\\
		&\le C^{\frac{1}{p}}\mathbbm{Q}\left[\sup_{t \le T} \tilde{B}_t^i \ge (k-1)\gamma C_\sigma^{-1}\right]^{\frac{p-1}{p}}
	\end{align*}
	where the second term in the last line follows from the Dubins-Schwarz Theorem, and $C$ is the constant from \cref{prop: BOUNDS ON THE 1 - P NORM OF THE RADON NIKODYM ESTIMATE}. By the Reflection Principle for Brownian motion, $\mathbbm{Q}\left[\sup_{t \le T} \tilde{B}_t^i \ge (k-1)\gamma C_\sigma^{-1}\right] = 2 \mathbbm{Q}\left[ \tilde{B}_T^i \ge (k-1)\gamma C_\sigma^{-1}\right]$. Therefore,
	\begin{equation*}
		\prob\left[\tau_k^i \le T\right]\le C^{\frac{1}{p}}\mathbbm{Q}\left[\sup_{t \le T} \tilde{B}_t^i \ge (k-1)\gamma C_\sigma^{-1}\right]^{\frac{p-1}{p}}\le  2^{\frac{p-1}{p}}C^{\frac{1}{p}}\mathbbm{Q}\left[\tilde{B}_T^i \ge (k-1)\gamma C_\sigma^{-1}\right]^{\frac{p-1}{p}}.
	\end{equation*}
	The claim now follows from employing the well-known result $\prob[\mathcal{N}(0,1) > x] \le \frac{1}{2}e^{-\frac{x^2}{2}}$ for any $x > 0$.
\end{proof}

\section{Concentration and Boundary Decay of the Particles}

By employing the Radon-Nikodym estimate from \cref{prop: BOUNDS ON THE 1 - P NORM OF THE RADON NIKODYM ESTIMATE}, we can derive a simple bound on the probability that $X^i$ is in any given set, conditioned on the event $\{X_t^i < 0\}$. Crucially, this bound holds uniformly in $i$ and $N$.

\begin{proposition}\label{prop: PROPOSTION SHOWING THAT WE HAVE AN UPPER BOUND ON THE PROBABILITYS OF XTI TAKING VALUES IN SOME SET}
	For any set $S \in \mathcal{B}(\R)$, $i = 1,\ldots,N$, and $\delta > 0$ sufficiently close to $0$, we have
	\begin{equation*}
		\prob\left[X_t^i \in S, X_t^i < 0\right] \le C t^{-\frac{\delta}{2}}\left[1 + \sum_{k \ge 1}\prob\left[\tau_k^i \le t\right]^\delta\right]\operatorname{Leb}(S)^{\delta},
	\end{equation*}
	where $C > 0$ is a constant that depends on $p,\,T,\,\mathfrak{K},\,\pref$, the initial distribution $X_0$, and the bounds on $b$ and $\sigma$, but is independent of $i$ and $N$.
\end{proposition}

\begin{proof}
	To begin, we fix $i = 1,\ldots,N$ and $t \in [0,T]$. Then, by the construction of $X_t^i$, $X_t^i < 0$ if and only if $t \in [0,\tau_1^i) \cup (\cup_{k \ge 1}[\tau_k^i + \varsigma_k^i, \tau_{k+1}^i))$. Therefore, by the law of total probability,
	\begin{equation}\label{eq: THE SUM DECOMPOSITION FOR XTI IN THE SPATIAL REGULARITY PROPOSITION}
		\prob\left[X_t^i \in S, X_t^i < 0\right] = \prob\left[X_t^i \in S, t < \tau_1^i\right] + \sum_{k \ge 1} \prob\left[X_t^i \in S, \tau_k^i + \varsigma_k^i\le t < \tau_{k +1}^i\right].
	\end{equation}
	The goal now is to apply the scale transform and employ the change of measure to rewrite the above as the probability of a Brownian motion with random jumps. As the scale transform is an injective function, we have
	\begin{align*}
		\prob\left[X_t^i \in S, X_t^i < 0\right] 
		=& \prob\left[\Upsilon_t(X_t^i) \in \Upsilon_t(S), t < \tau_1^i\right] \\
		&+ \sum_{k \ge 1} \prob\left[\Upsilon_t(X_t^i) \in \Upsilon_t(S), \tau_k^i + \varsigma_k^i\le t < \tau_{k +1}^i\right].
	\end{align*}
	We shall analyse each term above individually. On the event $\{t < \tau_1^i\}$, we have
	\begin{equation*}
		\Upsilon_t(X_t^i) = \Upsilon_0(X_0^i) + \tilde{B}_t, 
	\end{equation*}
	where $\tilde{B}$ is a $\mathbbm{Q}$-standard Brownian motion by \cref{prop: NOVIKOV CONDITION FOR THE STOCHASTIC EXPONENTIAL}. Therefore, by H\"older's inequality, we have
	\begin{equation}\label{eq: FIRST EQUATION IF THE SPATIAL REGULARITY OF X PROOF}
		\begin{aligned}
			\prob\left[\Upsilon_t(X_t^i) \in \Upsilon_t(S), t < \tau_1^i\right] &\le \prob\left[\Upsilon_0(X_0^i) + \tilde{B}_t \in \Upsilon_t(S)\right]\\
			&\le \E^\prob\left[\mathcal{E}_t^{1-p}\right]^{\frac{1}{p}}\mathbbm{Q}\left[\Upsilon_0(X_0^i) + \tilde{B}_t \in \Upsilon_t(S)\right]^{\frac{p-1}{p}}.
		\end{aligned}
	\end{equation}
	The right-hand side of the above is finite for all $p > 1$ sufficiently close to $1$ by \cref{prop: BOUNDS ON THE 1 - P NORM OF THE RADON NIKODYM ESTIMATE}. Girsanov's theorem, \citep[Chapter~3,Theorem~5.1]{karatzas2012brownian}, gives us that $\tilde{B}$ is a Brownian motion with respect to the same filtration of our original probability space. Therefore, as $X_0^i$ is $\mathcal{F}_0$-measurable and $\tilde{B}$ is independent of $\mathcal{F}_0$,
	\begin{align*}
		\mathbbm{Q}\left[\Upsilon_0(X_0^i) + \tilde{B}_t \in \Upsilon_t(S)\right] &= \int_\R \mathbbm{Q}\left[\Upsilon_0(x) + \tilde{B}_t \in \Upsilon_t(S)\right] \diff \nu_0(x)\\
		&= \int_\R \int_{\Upsilon_t(S) - \Upsilon_0(x)} (2 \pi t)^{-1/2}e^{-\frac{y^2}{2t}} \diff y \diff \nu_0(x)\\
		&= \int_\R \int_{S} (2 \pi t)^{-1/2}e^{-\frac{(\Upsilon_t(\tilde{y}) - \Upsilon_0(x))^2}{2t}} \Upsilon^\prime_t(\tilde{y}) \diff \tilde{y} \diff \nu_0(x)\\
		&\le C_\sigma (2 \pi t)^{-1/2} \operatorname{Leb}{(S)} \int_\R \diff \nu_0(x)\\
		&= C_\sigma (2 \pi t)^{-1/2} \operatorname{Leb}{(S)}
	\end{align*}
	where the third line follows from employing the substitution $\tilde{y} = \Upsilon_t^{-1}(y + \Upsilon_0(x))$, and the fourth line follows from upper bounding the exponential term and the $\Upsilon^\prime_t(\tilde{y})$ term. Setting $\delta = (p-1)p^{-1}$, by \eqref{eq: FIRST EQUATION IF THE SPATIAL REGULARITY OF X PROOF}, we have the bound
	\begin{equation}\label{eq: THE FRIST UPPERBOUND IN THE SPATIAL REGULARIY LEMMA}
		\prob\left[\Upsilon_t(X_t^i) \in \Upsilon_t(S), t < \tau_1^i\right] \le C_{\sigma,\delta}\E^\prob\left[\mathcal{E}_t^{1-p}\right]^{\frac{1}{p}}t^{-\frac{\delta}{2}}\operatorname{Leb}(S)^\delta.
	\end{equation}
	Now, we turn our attention to the second term. By construction, $\tau_k^i$ is independent of $\varsigma_k^i$ under the measure $\prob$, but independence need not be preserved under a change of measure. Hence, we apply the tower property of conditional expectation to preserve the independence between $\tau_k^i$ and $\varsigma_k^i$ and estimate conditional probabilities of $\tilde{B}$ under $\mathbbm{Q}$. We observe that on the event $\{\tau_k^i + \varsigma_k^i \le t < \tau_{k+1}^i\}$, we have
	\begin{equation*}
		\Upsilon_t(X_t^i) = \Upsilon_{\tau_k^i + \varsigma_k^i}(-\xi_k^i) + \tilde{B}_t - \tilde{B}_{\tau_k^i + \varsigma_k^i}. 
	\end{equation*}
	Therefore, 
	\begin{align*}
		&\prob\left[\Upsilon_t(X_t^i) \in \Upsilon_t(S), \tau_k^i + \varsigma_k^i\le t < \tau_{k +1}^i\right] \\
		&\hspace{4cm}\le \prob\left[\Upsilon_{\tau_k^i + \varsigma_k^i}(-\xi_k^i) + \tilde{B}_t - \tilde{B}_{\tau_k^i + \varsigma_k^i} \in \Upsilon_t(S), \tau_k^i + \varsigma_k^i\le t\right]. 
	\end{align*}
	Letting the event on the R.H.S. above be denoted by $A$, we deduce
	\begin{align}
		\prob\left[A\right] &= \E^\prob\left[\E^\prob\left[\ind_A \mid \tau_k^i + \varsigma_k^i, \xi_k^i\right]\right] \notag \\
		&= \E^\prob\left[\E^{\mathbbm{Q}}\left[\mathcal{E}_t^{-1}\ind_A \mid \tau_k^i + \varsigma_k^i, \xi_k^i\right]\right] \notag\\
		&\le \E^\prob\left[\E^{\prob}\left[\mathcal{E}_t^{1-p} \mid \tau_k^i + \varsigma_k^i, \xi_k^i\right]^{\frac{1}{p}}\E^{\mathbbm{Q}}\left[\ind_A \mid \tau_k^i + \varsigma_k^i, \xi_k^i\right]^{\frac{p-1}{p}}\right]\notag\\
		&\le \E^\prob\left[\mathcal{E}_t^{1-p}\right]^{\frac{1}{p}}\E^\prob\left[\E^{\mathbbm{Q}}\left[\ind_A \mid \tau_k^i + \varsigma_k^i, \xi_k^i\right]\right]^{\frac{p-1}{p}}, \label{eq: THE SECOND EQUATION IN THE SPATIAL REGULARTY PROPOSITION}
	\end{align}
	where the penultimate line follows from the Conditional H\"older inequality and the last line follows from H\"older's inequality. The right-hand side of the above is finite for all $p > 1$ sufficiently close to $1$ by \cref{prop: BOUNDS ON THE 1 - P NORM OF THE RADON NIKODYM ESTIMATE}. 
	
	We now turn our attention to estimating the second term in the product found in \eqref{eq: THE SECOND EQUATION IN THE SPATIAL REGULARTY PROPOSITION}. Let $\tau$ be any stopping time such that $0 < \tau \le t$ almost surely and $\xi$ be an $\mathcal{F}_0$-measurable random variable supported on $(0,R]$ for some $R > 0$. Furthermore, let $\tau^n$ ($\xi^n$) denote the discrete approximation to $\tau$ ($\xi$, respectively) defined by
	\begin{align*}
		\tau^n &=\sum_{j = 1}^{2^n} tj2^{-n} \ind_{((j-1)t2^{-n},jt2^{-n}]}(\tau),\\
		\xi^n &= \sum_{l = 1}^{2^n} Rl2^{-n} \ind_{((l-1)R2^{-n},lR2^{-n}]}(\xi).
	\end{align*}
	Then for any open sets $A,B,C \subset \R$,
	\begin{align*}
		&\mathbbm{Q}\left[\Upsilon_{\tau^n}(-\xi^n) + \tilde{B}_t - \tilde{B}_{\tau^n} \in A, \tau^n \in B, \xi^n \in C\right]\\
		& = \sum_{j,l = 1}^{2^n} \mathbbm{Q}\left[\Upsilon_{\tau^n}(-\xi^n) + \tilde{B}_t - \tilde{B}_{\tau^n} \in A, \tau^n = tj2^{-n}, \xi^n = Rl2^{-n} \right] \ind_B(tj2^{-n}) \ind_C(Rl2^{-n})\\
		& = \sum_{j,l = 1}^{2^n} \mathbbm{Q}\left[\Upsilon_{\tau^n}(-\xi^n) + \tilde{B}_t - \tilde{B}_{\tau^n} \in A \right] \ind_B(tj2^{-n}) \ind_C(Rl2^{-n})\mathbbm{Q}[\tau^n = tj2^{-n}, \xi^n = Rl2^{-n}]\\
		& = \int_{\R^2} \mathbbm{Q}\left[\Upsilon_{s}(-x) + \tilde{B}_t - \tilde{B}_{s} \in A \right] \ind_B(s) \ind_C(x)\mathbbm{Q}\left[\tau^n \in \diff s, \xi^n \in \diff x\right]\\
		& = \int_{\Omega} \mathbbm{Q}\left[\Upsilon_{\tau^n(\omega)}(-\xi^n(\omega)) + \tilde{B}_t - \tilde{B}_{\tau^n(\omega)} \in A \right] \ind_B(\tau^n(\omega)) \ind_C(\xi^n(\omega))\diff\mathbbm{Q}\left[\omega\right] 
	\end{align*}
	Therefore, by first taking limits as $n \to  \infty$ and then applying a monotone class theorem argument, we have
	\begin{equation}\label{eq: THE THIRD EQUATION IN THE SPATIAL REGULARITY PROPOSITION}
		\mathbbm{Q}\left[\Upsilon_{\tau}(-\xi) + \tilde{B}_t - \tilde{B}_{\tau} \in A \mid \tau, \xi \right] = \mathbbm{Q}\left[\Upsilon_{\tau(\cdot)}(-\xi(\cdot)) + \tilde{B}_t - \tilde{B}_{\tau(\cdot)} \in A\right]
	\end{equation}
	for all measurable sets $A$. Girsanov's theorem, \citep[Chapter~3,Theorem~5.1]{karatzas2012brownian}, gives us that $\tilde{B}$ is a Brownian motion with respect to the same filtration as our original probability space. As $\xi_k^i$ and $\varsigma_k^i$ are $\mathcal{F}_0$-measurable and $\tau_k^i + \varsigma_k^i$ is a stopping time, returning to the second term in \eqref{eq: THE SECOND EQUATION IN THE SPATIAL REGULARTY PROPOSITION} and employing \eqref{eq: THE THIRD EQUATION IN THE SPATIAL REGULARITY PROPOSITION}, we have
	\begin{align*}
		&\E^\prob\left[\E^{\mathbbm{Q}}\left[\ind_A \mid \tau_k^i + \varsigma_k^i, \xi_k^i\right]\right] \\
		&\hspace{2cm}= \E^{\prob}\left[\mathbbm{Q}\left[\Upsilon_{\tau_k^i + \varsigma_k^i}(-\xi_k^i) + \tilde{B}_t - \tilde{B}_{\tau_k^i + \varsigma_k^i} \in \Upsilon_t(S), \tau_k^i + \varsigma_k^i\le t\mid \tau_k^i + \varsigma_k^i, \xi_k^i\right]\right]\\
		&\hspace{2cm}= \int_{[0,t]\times \R} \mathbbm{Q}\left[\Upsilon_{s}(-x) + \tilde{B}_t - \tilde{B}_{s} \in \Upsilon_t(S)\right]\prob\left[\tau_k^i + \varsigma_k^i \in \diff s, \xi_k^i \in \diff x \right] \\
		&\hspace{2cm}\le C_\sigma(2\pi)^{-1/2}\operatorname{Leb}(S) \int_0^t(t- s)^{-1/2}\prob\left[\tau_k^i + \varsigma_k^i \in \diff s \right] \\
		&\hspace{2cm}= C_\sigma(2\pi)^{-1/2}\operatorname{Leb}(S) \int_0^t\int_0^s(t- s)^{-1/2} \pref(s-u)\prob\left[\tau_k^i \in \diff u \right] \diff s 
	\end{align*}
	where the third line follows from $\mathbbm{Q}[\Upsilon_{s}(-x) + \tilde{B}_t - \tilde{B}_{s} \in \Upsilon_t(S)] \le C_\sigma(2\pi)^{-1/2}(t- s)^{-1/2}\operatorname{Leb}(S)$ and the last line is due to the independence between $\tau_k^i$ and $\varsigma_k^i$. Lastly, by first applying Fubini's Theorem followed by using the fact that $\pref$ is continuous, we compute
	\begin{align*}
		&\E^\prob\left[\E^{\mathbbm{Q}}\left[\ind_A \mid \tau_k^i + \varsigma_k^i, \xi_k^i\right]\right] \\
		&\hspace{2cm}\le C_\sigma(2\pi)^{-1/2}\operatorname{Leb}(S) \int_0^t\int_u^t(t- s)^{-1/2} \pref(s-u) \diff s \prob\left[\tau_k^i \in \diff u \right] \\
		&\hspace{2cm}\le C_\sigma (2\pi)^{-1/2}\norm{\pref}_\infty\operatorname{Leb}(S) \int_0^t\int_u^t(t- s)^{-1/2} \diff s \prob\left[\tau_k^i \in \diff u \right] \\
		&\hspace{2cm}= 2C_\sigma(2\pi)^{-1/2}\norm{\pref}_\infty\operatorname{Leb}(S) \int_0^t(t- u)^{1/2} \prob\left[\tau_k^i \in \diff u \right] \\
		&\hspace{2cm}\le 2C_\sigma(2\pi)^{-1/2}\norm{\pref}_\infty t^{1/2}\prob\left[\tau_k^i \le t\right]\operatorname{Leb}(S). 
	\end{align*}
	Returning to \eqref{eq: THE SECOND EQUATION IN THE SPATIAL REGULARTY PROPOSITION}, and increasing $C_\sigma$ when necessary, we have
	\begin{equation}\label{eq: THE FOURTH EQUATION IN THE SPATILA REGULARITY PROPOSITION}
		\begin{aligned}
			&\prob\left[\Upsilon_t(X_t^i) \in \Upsilon_t(S), \tau_k^i + \varsigma_k^i\le t < \tau_{k +1}^i\right] \\
			&\hspace{3cm}\le C_{\sigma,\delta}\norm{\pref}^\delta_\infty t^{\delta/2}\E^\prob\left[\mathcal{E}_t^{1-p}\right]^{\frac{1}{p}}\prob\left[\tau_k^i \le t\right]^\delta\operatorname{Leb}(S)^{\delta},
		\end{aligned}
	\end{equation}
	where $\delta = (p-1)p^{-1}$. Combining the upper bounds in \eqref{eq: THE FRIST UPPERBOUND IN THE SPATIAL REGULARIY LEMMA} and \eqref{eq: THE FOURTH EQUATION IN THE SPATILA REGULARITY PROPOSITION} into \eqref{eq: THE SUM DECOMPOSITION FOR XTI IN THE SPATIAL REGULARITY PROPOSITION}, we obtain
	\begin{align*}
		\prob\left[X_t^i \in S, X_t^i < 0\right] \le& C_{\sigma,\delta}\E^\prob\left[\mathcal{E}_t^{1-p}\right]^{\frac{1}{p}}t^{-\frac{\delta}{2}}\operatorname{Leb}(S)^\delta\\
		&+ \sum_{k \ge 1} C_{\sigma,\delta}\norm{\pref}^\delta_\infty t^{\delta/2}\E^\prob\left[\mathcal{E}_t^{1-p}\right]^{\frac{1}{p}}\prob\left[\tau_k^i \le t\right]^\delta\operatorname{Leb}(S)^{\delta}\\
		\le& C_{\sigma,\delta,T,\pref} t^{-\frac{\delta}{2}}\E^\prob\left[\mathcal{E}_t^{1-p}\right]^{\frac{1}{p}}\left[1 + \sum_{k \ge 1}\prob\left[\tau_k^i \le t\right]^\delta\right]\operatorname{Leb}(S)^{\delta}.
	\end{align*}
	The right-hand side of the above is finite for all $p > 1$ sufficiently close to $1$ by \cref{prop: BOUNDS ON THE 1 - P NORM OF THE RADON NIKODYM ESTIMATE}.
\end{proof}

The above is a rough bound that in particular ensures
\[
\mathbb{E}[\nu^N_t(S)] \leq C t^{-\frac{\delta}{2}}\text{Leb}(S)^\delta
\]
for all $t\in(0,T]$ and all $S\in\mathcal{B}(\mathbb{R})$, for some $\delta\in (0,1)$ and $C>0$ that is uniform in $N\geq 1$. Near the origin, we expect much better control, due to absorption. This is quantified by the next result.

\begin{proposition}[Boundary decay]\label{prop: BOUNDARY DECAY OF THE EMPIRICAL MEAUSRES IN EXPECTATION}
	There exists a $\delta \in (0,1]$ and $\beta > 0$ such that uniformly in $N$ and $t \in (0,T]$, we have 
	\begin{equation*}
		\E\left[\nu_t^N(-\varepsilon,0)\right] = t^{-\delta/2}O(\varepsilon^{ 1 + \beta})\quad \text{as}\quad \varepsilon \to  0.
	\end{equation*}
\end{proposition}

The estimate will follow from controlling $\prob(X_t^i \in (-\varepsilon, 0),\, X_t^i < 0)$, uniformly in $i$. The strategy is as follows. By the definition of the scale transform in \cref{prop: SCALE TRANSFORM LEMMA }, the function $x \mapsto \Upsilon_t(x)$ is a bijection for any fixed $t$. In addition, $\Upsilon_t(x) \le 0$ if and only if $x \le 0$. Consequently, we have
\begin{equation*}
	\prob\left[X_t^i \in S,\, X_t^i < 0\right] = \prob\left[\Hat{X}_t \in S_t,\, \Hat{X}_t < 0\right],
\end{equation*}
where $\Hat{X}_t = \Upsilon_t(X_t^i)$ and $S_t = \Upsilon_t(S)$. By applying a change of measure, we may view $\Hat{X}$ as a neuronal particle driven solely by a standard Brownian motion. This particle is paused at $0$ for a random duration $\varsigma_k^i$ after reaching $0$ for the $k^{\textnormal{th}}$ time. Subsequently, it is then reset to a random value in the negative half-line after this waiting period.

We then approximate $\Hat{X}$ by allowing it to evolve up to time $s$, followed by running an independent absorbed Brownian motion for time $t - s$. By employing the bounds on the density of an absorbed Brownian motion, we may estimate the probability that this approximate particle is in $S_t$. Lastly, by sending $s$ towards $t$ from below, we may obtain an estimate on the probability that $\Hat{X}$ is in $S_t$. This strategy mirrors the approach in \citep[Section~6.3]{hambly2019spde}.

\begin{proof}
	By \cref{prop: SCALE TRANSFORM LEMMA },
	\begin{align}
		\notag \diff\Hat{X}_t &= \hat{b}_t^i \ind_{\{\Hat{X}_t < 0\}}\diff{t} + \ind_{\{\Hat{X}_t < 0\}} \diff{W_t^i} + \diff\sum_{k \ge 1} \Upsilon_{\tau_k^i + \varsigma_k^i}(-\xi_k^i)\ind_{[0,t]}(\tau_k^i + \varsigma_k^i), \\
		\label{eq: EQUATION OF THE NEURONAL PARTICAL DRIVEN BY A BROWNIAN MOTION AND RANDOM RESETS}&= \ind_{\{\Hat{X}_t < 0\}} \diff \tilde{B}_t + \diff\sum_{k \ge 1} \Upsilon_{\tau_k^i + \varsigma_k^i}(-\xi_k^i)\ind_{[0,t]}(\tau_k^i + \varsigma_k^i)
	\end{align}
	where $\tilde{B}$ is a $\mathbb{Q}$-Brownian motion. Take $S:=(-\varepsilon ,0)$, for an arbitrary $\varepsilon>0$. As in the outline of our proof strategy above, we set $S_t:= \Upsilon_t(S)$ and $\hat{X}_t:=\Upsilon_t(X_t^i)$. Consider a given interval $[0,t]$ and define
	\begin{equation*}
		u(s,x) = \int_{S_t} G_{t - s}(y,x) \diff y
	\end{equation*}
	for $s \in [0,t]$, where $G_t(y,x) = p_t(x-y) - p_t(x+y)$, with $p_t(x) = (2\pi t)^{-1/2}e^{-x^2/2t}$, is the Dirichlet heat kernel on the domain $(-\infty,0)$. We remark that $u$ is a classical solution to the terminal-boundary value problem
	\begin{equation}\label{eq: PDE THAT U SOLVES}
		\begin{cases}
			\begin{aligned}
				&\partial_s u(s,x)  + \frac{1}{2} \Delta u(s,x) = 0 &\quad\textnormal{on}\quad &[0,t) \times (-\infty,0),\\
				&u(t,x) = \ind_{S_t}(x) &\quad\textnormal{on}\quad &\{t\} \times (-\infty,0),\\
				&u(s,0) = 0 &\quad\textnormal{on}\quad &[0,t) \times \{0\}.
			\end{aligned}
		\end{cases}
	\end{equation}
	
	Now, fix any $t_0 < t$ and let
	\begin{equation*}
		v(s) \coloneqq \E \left[u(s,\Hat{X}_s)\right].
	\end{equation*}
	Applying It\^o's formula with jumps to $u(s,\Hat{X}_s)$ and utilizing \eqref{eq: PDE THAT U SOLVES}, we obtain
	\begin{equation}\label{eq:THE DECOMPOSITION WE WISH TO CONTROL IN THE BOUNDARY ESTIMATES}
		v(t_0) = v(0) + \int_0^{t_0} \E\left[\hat{b}_s^i\ind_{\{\Hat{X}_s < 0\}}\partial_x u(s,\Hat{X}_s)\right] \diff s + \E \left[\sum_{s \le t_0} u(s,\Hat{X}_s) - u(s,\Hat{X}_{s-})\right].
	\end{equation}
	The underlying idea is that by taking the limit $t_0 \uparrow t$, the term on the left-hand side converges to $\prob[\Hat{X}_t \in S_t]$, which is our quantity of interest. Thus, we need to obtain control of the three terms on the right-hand side. The first term equates to $\prob[\Upsilon_0(X_0^i) + B_t \in S_t,\, t < \tau^B]$ for some standard Brownian motion $B$ independent of $X_0^i$, where $\tau^B \coloneqq \inf\{s \ge 0\,:\, \Upsilon_0(X_0^i) + B_s \ge 0\}$. As $S_t \approx (-\varepsilon,0)$, it follows that
	\begin{equation}\label{eq:BOUND ON THE FIRST TIME IN THE DECOMPOSITION THAT WE CARE ABOUT}
		\prob[\Upsilon_0(X_0^i) + B_t \in S_t,\, t < \tau^B] = O(t^{-1/2}\varepsilon^2)    
	\end{equation}
	uniformly in $i$.
	
	To estimate the third term on the right-hand side of \eqref{eq:THE DECOMPOSITION WE WISH TO CONTROL IN THE BOUNDARY ESTIMATES}, we begin by noting that if $s$ is a reset time, i.e. $ s = \tau_k^i + \varsigma_k^i$, $\Hat{X}_{s-} = \Upsilon_s (0) = 0$ and $\Hat{X}_s = \Upsilon_{\tau_k^i + \varsigma_k^i} (-\xi_k^i)$. Therefore, we can express
	\begin{equation}\label{eq:REPRESENTATION OF THE THIRD TERM WITH STOPPING TIMES}
		\E \left[\sum_{s \le t_0} u(s,\Hat{X}_s) - u(s,\Hat{X}_{s-})\right] = \sum_{k \ge 1} \E \left[\ind_{t_0 \ge \tau_k^i + \varsigma_k^i} u(\tau_k^i + \varsigma_k^i,\Upsilon_{\tau_k^i + \varsigma_k^i}(-\xi_k^i))\right].
	\end{equation}
	Proceeding, we find that
	\begin{align*}
		&\E \left[\ind_{t_0 \ge \tau_k^i + \varsigma_k^i} u(\tau_k^i + \varsigma_k^i,\Upsilon_{\tau_k^i + \varsigma_k^i}(-\xi_k^i))\right]\\
		&\qquad= \int_{\gamma}^R \int_0^{t_0}\int_{S_t} G_{t - s}(y,\Upsilon_s(-x)) \diff y \prob[\tau_k^i + \varsigma_k^i \in \diff s] \prob[\xi_k^i \in \diff x] \\ \intertext{Applying the change of variables $\tilde{y} = - y$ and using the fact that $G_t(-y,-x) = G_t(y,x)$ for any $x,y$,}
		&\qquad= \int_{\gamma}^R \int_0^{t_0}\int_{-S_t} G_{t - s}(\tilde{y},-\Upsilon_s(-x)) \diff \tilde{y} \prob[\tau_k^i + \varsigma_k^i \in \diff s] \prob[\xi_k^i \in \diff x] \\ \intertext{Using the change of variables $\tilde{y} = -\Upsilon_t(-\tilde{x})$,}
		&\qquad= \int_{\gamma}^R \int_0^{t_0}\int_0^\varepsilon G_{t - s}(-\Upsilon_t(-\tilde{x}),-\Upsilon_s(-x)) \Upsilon_t^\prime(-\tilde{x})\diff \tilde{x} \prob[\tau_k^i + \varsigma_k^i \in \diff s] \prob[\xi_k^i \in \diff x].
	\end{align*}
	
	By \citep[Lemma~2.2.1]{sojmark_2019}, for any $x,y \ge 0$, we have
	\begin{equation}\label{eq:BOUND ON THE DENSITY OF KILLED BM FROM AS THESIS}
		G_t(x,y) \le \frac{C}{\sqrt{t}}\left(\frac{x}{\sqrt{t}} \wedge 1\right) \left(\frac{y}{\sqrt{t}} \wedge 1\right) e^{-\frac{(x-y)^2}{4t}},
	\end{equation}
	for some constant $C > 0$. Furthermore, by definition of $\Upsilon_t$, we may find $c,C > 0$ such that
	\begin{equation}\label{eq:GENERAL BOUNDS ON THE SCALE TRANSFORM}
		\nnnorm{\Upsilon_t(-\tilde{x})} \le C \nnnorm{\tilde{x}}, \qquad \nnnorm{\Upsilon_t(-\tilde{x}) - \Upsilon_t(-x)} \ge c(\gamma - \varepsilon), \qquad \nnnorm{\Upsilon_t^\prime(-\tilde{x})} \le C,
	\end{equation}
	for all $\tilde{x} \in (0, \varepsilon)$ and $x \in [\gamma,R]$. Employing \eqref{eq:BOUND ON THE DENSITY OF KILLED BM FROM AS THESIS} and \eqref{eq:GENERAL BOUNDS ON THE SCALE TRANSFORM},
	\begin{align}
		\notag \E \left[\ind_{t_0 \ge \tau_k^i + \varsigma_k^i} u(\tau_k^i + \varsigma_k^i,\Upsilon_{\tau_k^i + \varsigma_k^i}(-\xi_k^i))\right] &\lesssim \int_0^{t_0}\int_0^\varepsilon \frac{\tilde{x}}{t-s} e^{-\frac{c(\gamma - \varepsilon)^2}{4(t-s)}} \diff \tilde{x} \prob[\tau_k^i + \varsigma_k^i \in \diff s]\\
		&\lesssim (\gamma - \varepsilon)^{-2} \int_0^{t_0}\int_0^\varepsilon {\tilde{x}}\diff \tilde{x} \prob[\tau_k^i + \varsigma_k^i \in \diff s] \label{eq: BOUND ON ONE OF THE ELEMENTS IN THE INFINITE SUM}
	\end{align}
	where the final line follows from the fact that for any fixed $z > 0$, the function $t \mapsto t^{-1}e^{-z^2/4t}$ is bounded by $4e^{-1}z^{-2}$. Substituting \eqref{eq: BOUND ON ONE OF THE ELEMENTS IN THE INFINITE SUM} into \eqref{eq:REPRESENTATION OF THE THIRD TERM WITH STOPPING TIMES},
	\begin{equation}\label{eq:BOUND ON THE THIRD TIME IN THE DECOMPOSITION THAT WE CARE ABOUT}
		\E \left[\sum_{s \le t_0} u(s,\Hat{X}_s) - u(s,\Hat{X}_{s-})\right]  \le C \varepsilon^2 \sum_{k \ge 1} \prob[\tau_k^i \le T],
	\end{equation}
	for all $\varepsilon \ll 1$ and some constant $C$ independent of $i$.
	
	Now we turn our attention to the middle term. By definition of $u$,
	\begin{align}
		&\int_0^{t_0}\E\left[\nnorm{\hat{b}_s^i}\ind_{\{\Hat{X}_s < 0\}}\nnorm{\partial_x u(s,\Hat{X}_s)}\right] \diff s  \\
		&\hspace{3cm} \le \int_0^t \int_{S_t} \E\left[\nnorm{\hat{b}_s^i}\ind_{\{\Hat{X}_s < 0\}}\nnorm{\partial_x G_{t-s}(y,\Hat{X}_s)}\right] \diff y \diff s \notag \\
		&\hspace{3cm} \le \int_0^t\int_{S_t} \E\left[\nnorm{\hat{b}_s^i}\ind_{\{s < \hat{\tau}_1\}}\nnorm{\partial_x G_{t-s}(y,\Hat{X}_s)}\right] \diff y \diff s \notag \\
		&\hspace{3cm}+ \sum_{k \ge 1} \int_0^t\int_{S_t} \E\left[\nnorm{\hat{b}_s^i}\ind_{\{\hat{\tau}_k + \varsigma_k^i \le s < \hat{\tau}_{k+1}\}}\nnorm{\partial_x G_{t-s}(y,\Hat{X}_s)}\right] \diff y \diff s \notag \\
		&\hspace{3cm}= I + \sum_{k \ge 1} I_k, \label{eq: DECOMPOSITION OF THE MIDDLE TERM INTO AN INFINITE SUM OF INTEGRALS}
	\end{align}
	where $\hat{\tau}_k = \inf\{ t > \hat{\tau}_{k-1} + \varsigma_{k-1}^i\,:\, \hat{X}_{t-} \ge 0\}$ with $\hat{\tau}_0 = \varsigma_0^i = 0$. To begin with $I$, by applying the tower property and Fubini's Theorem, we have
	\begin{equation}\label{eq: APPLICATION OF THE TOWER PROPERTY TO I IN OUR DECOMPOSITION OF THE MIDDLE TERM}
		I = \E \int_0^t\int_{S_t} \E\left[\left.\nnorm{\hat{b}_s^i}\ind_{\{s < \hat{\tau}_1\}}\nnorm{\partial_x G_{t-s}(y,\Hat{X}_s)}\right|X_0^i\right] \diff y \diff s.
	\end{equation}
	Now for any $p,q > 1$, we may apply H\"older's inequality to the conditional expectation in \eqref{eq: APPLICATION OF THE TOWER PROPERTY TO I IN OUR DECOMPOSITION OF THE MIDDLE TERM}. This yields
	\begin{equation}\label{eq: APPLICATION OF HOLDERS INEQUALITY TO I IN OUR DECOMPOSITION OF THE MIDDLE TERM}
		\begin{aligned}
			I \le \E \int_0^t\int_{S_t} &\E\left[\left.\nnorm{\hat{b}_s^i}^{\frac{q}{q-1}}\right|X_0^i\right]^{\frac{q-1}{q}} \E\left[\left.\mathcal{E}_T^{1-p}\right|X_0^i\right]^{\frac{1}{pq}} \\
			&\times\E^{\mathbb{Q}}\left[\left.\nnorm{\partial_x G_{t-s}(y,\Hat{X}_s)}^{\frac{pq}{p-1}}\ind_{\{s < \hat{\tau}_1\}}\right|X_0^i\right]^{\frac{p-1}{pq}}\diff y \diff s.
		\end{aligned}
	\end{equation}
	By \cref{prop: SCALE TRANSFORM LEMMA }, we deduce that $\nnorm{\hat{b}_s^i} \le C(1 + \hat{\Lambda}_s^i) \le C(1 + \hat{\Lambda}_T^i) $. Therefore,
	\begin{equation}\label{eq: APPLICATION OF HOLDERS INEQUALITY TO I IN OUR DECOMPOSITION OF THE MIDDLE TERM AND MAKING THE FIRST TWO TERMS INDEPENDENT OF S}
		\begin{aligned}
			I \lesssim \E \int_0^t\int_{S_t} &\E\left[\left.\nnorm{1 + \hat{\Lambda}_T^i}^{\frac{q}{q-1}}\right|X_0^i\right]^{\frac{q-1}{q}} \E\left[\left.\mathcal{E}_T^{1-p}\right|X_0^i\right]^{\frac{1}{pq}} \\
			&\times\E^{\mathbb{Q}}\left[\left.\nnorm{\partial_x G_{t-s}(y,\Hat{X}_s)}^{\frac{pq}{p-1}}\ind_{\{s < \hat{\tau}_1\}}\right|X_0^i\right]^{\frac{p-1}{pq}}\diff y \diff s.
		\end{aligned}
	\end{equation}
	Since the sum of the exponents $(q-1)/q + 1/pq + (p-1)/pq = 1$ and the last term is the only term that depends on $y$ and $s$, applying H\"older's inequality one more time to \eqref{eq: APPLICATION OF HOLDERS INEQUALITY TO I IN OUR DECOMPOSITION OF THE MIDDLE TERM AND MAKING THE FIRST TWO TERMS INDEPENDENT OF S} yields
	\begin{equation}\label{eq: APPLICATION OF HOLDERS INEQUALITY TO I AGAIN IN OUR DECOMPOSITION OF THE MIDDLE TERM}
		\begin{aligned}
			I \lesssim &\E\left[\nnorm{1 + \hat{\Lambda}_T^i}^{\frac{q}{q-1}}\right]^{\frac{q-1}{q}} \E\left[\mathcal{E}_T^{1-p}\right]^{\frac{1}{pq}} \\
			&\times\E\left[ \left(\int_0^t\int_{S_t} \E^{\mathbb{Q}}\left[\left.\nnorm{\partial_x G_{t-s}(y,\Hat{X}_s)}^{a}\ind_{\{s < \hat{\tau}_1\}}\right|X_0^i\right]^{\frac{1}{a}}\diff y \diff s\right)^a\right]^\frac{1}{a},
		\end{aligned}
	\end{equation}
	where $a = pq/(p-1)$. By \cref{cor: SUBGUASSIANTY COROLLARY}, the first term in \eqref{eq: APPLICATION OF HOLDERS INEQUALITY TO I AGAIN IN OUR DECOMPOSITION OF THE MIDDLE TERM} is bounded for all $q$ uniformly in $N$ and the second term is bounded for all $p$ close enough to $1$ by \cref{prop: BOUNDS ON THE 1 - P NORM OF THE RADON NIKODYM ESTIMATE}. Thus, it remains to establish control over the last term in  \eqref{eq: APPLICATION OF HOLDERS INEQUALITY TO I AGAIN IN OUR DECOMPOSITION OF THE MIDDLE TERM}. 
	
	To this end, we shall employ the techniques in \citep[Chapter~2]{sojmark_2019}. For any $x_0 < 0$, we have $\mathbb{Q}[\Hat{X}_s \in \cdot,\, s < \hat{\tau}_1\mid X_0^i = x_0] = \mathbb{Q}[\Upsilon_0(x_0) + \tilde{B}_s\in \cdot,\, s < \hat{\tau}_1 \mid X_0^i = x_0]$. This is the law of a stopped Brownian Motion starting at $\Upsilon_0(x_0)$. Therefore,
	\begin{equation}\label{eq: CONTROL ON THE MIDDLE TERM ONE}
		\begin{aligned}
			&\int_{S_t} \E^{\mathbb{Q}}\left[\left.\nnorm{\partial_x G_{t-s}(y,\Hat{X}_s)}^{a}\ind_{\{s < \hat{\tau}_1\}}\right|X_0^i = x_0\right]^{\frac{1}{a}}\diff y\\
			&\hspace{4cm} = \int_{S_t} \left(\int_{-\infty}^0 \nnorm{\partial_x G_{t-s}(y,x)}^{a}G_s(z,x)\diff x\right)^{\frac{1}{a}} \diff y,
		\end{aligned}
	\end{equation}
	where $z = \Upsilon_0(x_0)$. By applying the change of variables $\tilde{x} = -x$ and $ \tilde{y} = -y$, and considering the expression for $G_t(y,x)$, we get
	\begin{equation}\label{eq: CONTROL ON THE MIDDLE TERM TWO}
		\begin{aligned}
			&\int_{S_t} \E^{\mathbb{Q}}\left[\left.\nnorm{\partial_x G_{t-s}(y,\Hat{X}_s)}^{a}\ind_{\{s < \hat{\tau}_1\}}\right|X_0^i = x_0\right]^{\frac{1}{a}}\diff y \\
			&\hspace{4cm}= \int_{-S_t} \left(\int_{0}^{\infty} \nnorm{\partial_{\tilde{x}} G_{t-s}(\tilde{y},\tilde{x})}^{a}G_s(\nnorm{z},\tilde{x})\diff \tilde{x}\right)^{\frac{1}{a}} \diff \tilde{y},
		\end{aligned}
	\end{equation}
	Writing out the expression for $\partial_{\tilde{x}} G_{t-s}$ and $G_s$, and using the upper bound
	\begin{equation*}
		e^{-\frac{(\tilde{x} - \nnorm{z})^2}{2s}} - e^{-\frac{(\tilde{x} + \nnorm{z})^2}{2s}} \le \left(\frac{2\tilde{x}\nnorm{z}}{s}\wedge 1\right)e^{-\frac{(\tilde{x} - \nnorm{z})^2}{2s}},
	\end{equation*}
	we obtain $\int_{S_t} \E^{\mathbb{Q}}\left[\left.\nnorm{\partial_x G_{t-s}(y,\Hat{X}_s)}^{a}\ind_{\{s < \hat{\tau}_1\}}\right|X_0^i = x_0\right]^{\frac{1}{a}}\diff y \lesssim (t - s)^{-1/2}s^{-1/2a} \tilde{I}(s)$, where
	\begin{equation}\label{eq:EXPRESSION OF I TILDE}
		\tilde{I}(s) \coloneqq \int_{-S_t} \left(\int_{0}^{\infty} \nnnorm{\frac{(\tilde{y} - \tilde{x})}{t - s}e^{-\frac{(\tilde{y} - \tilde{x})^2}{2(t - s)}} + \frac{(\tilde{y} + \tilde{x})}{t - s}e^{-\frac{(\tilde{y} + \tilde{x})^2}{2(t - s)}}}^{a}\frac{\tilde{x}\nnorm{z}}{s}e^{-\frac{(\tilde{x} - \nnorm{z})^2}{2s}}\diff \tilde{x}\right)^{\frac{1}{a}} \diff \tilde{y}
	\end{equation}
	
	Now \eqref{eq:EXPRESSION OF I TILDE} mirrors the form of $ \tilde{I}(s)$ as delineated in the proof of \citep[Proposition~2.4.3]{sojmark_2019}. Therefore, we have the bound
	\begin{equation}\label{eq: CONTROL ON THE MIDDLE TERM THREE}
		\int_0^t\int_{S_t} \E^{\mathbb{Q}}\left[\left.\nnorm{\partial_x G_{t-s}(y,\Hat{X}_s)}^{a}\ind_{\{s < \hat{\tau}_1\}}\right|X_0^i = x_0\right]^{\frac{1}{a}}\diff y \diff s \le C_a \int_{-S_t} t^{-\frac{1}{a}} \nnorm{z}^{\frac{1}{a}} \tilde{y}^{\frac{1}{a}}e^{-\frac{(\tilde{y} - \nnorm{z})^2}{4at}} \diff \tilde{y}.
	\end{equation}
	Recalling $S_t = \Upsilon_t(S)$ and $\Upsilon_0(x_0) = z$, employing the change of variable $\tilde{y} = - \Upsilon_t(-\tilde{x})$ and the upper bounds in \eqref{eq:GENERAL BOUNDS ON THE SCALE TRANSFORM} which hold for all $\tilde{x}$, we have
	\begin{align*}
		\int_0^t\int_{S_t} \E^{\mathbb{Q}}\left[\left.\nnorm{\partial_x G_{t-s}(y,\Hat{X}_s)}^{a}\ind_{\{s < \hat{\tau}_1\}}\right|X_0^i = x_0\right]^{\frac{1}{a}}\diff y \diff s 
		&\le C_{a,\sigma} \int_0^\varepsilon t^{-\frac{1}{a}} \nnorm{x_0}^{\frac{1}{a}} \tilde{x}^{\frac{1}{a}} \diff \tilde{x} \\
		&= C_{a,\sigma} t^{-\frac{1}{a}}\nnorm{x_0}^{\frac{1}{a}} \varepsilon^{1 + \frac{1}{a}}.
	\end{align*}
	Returning to \eqref{eq: APPLICATION OF HOLDERS INEQUALITY TO I AGAIN IN OUR DECOMPOSITION OF THE MIDDLE TERM}, using the upper bound above,
	\begin{equation}\label{eq:BOUND ON THE I TERM IN THE DECOMPOSITION OF THE }
		I \le C_{a,p,q,\sigma} t^{-\frac{1}{a}}\varepsilon^{1 + \frac{1}{a}}\E\left[\nnorm{X_0}\right]^{\frac{1}{a}}
	\end{equation}
	
	Now considering the term $I_k$ in \eqref{eq: DECOMPOSITION OF THE MIDDLE TERM INTO AN INFINITE SUM OF INTEGRALS}, performing analogous computations as in \eqref{eq: APPLICATION OF THE TOWER PROPERTY TO I IN OUR DECOMPOSITION OF THE MIDDLE TERM} - \eqref{eq: APPLICATION OF HOLDERS INEQUALITY TO I AGAIN IN OUR DECOMPOSITION OF THE MIDDLE TERM} but conditioning on $\hat{\tau}_k^i + \varsigma_k^i$ and $\xi_k^i$, we obtain
	\begin{equation}\label{eq: APPLICATION OF FUBINI AND HOLDER TO GET A PRELIMINARY UPPERBOUND ON IK}
		I_k \lesssim C_{p,q} \E\left[ \left(\int_0^t\int_{S_t} \E^{\mathbb{Q}}\left[\left.\nnorm{\partial_x G_{t-s}(y,\Hat{X}_s)}^{a}\ind_{\{\hat{\tau}_k + \varsigma_k^i \le s < \hat{\tau}_{k+1}\}}\right|\hat{\tau}_k^i + \varsigma_k^i,\xi_k^i\right]^{\frac{1}{a}}\diff y \diff s\right)^a\right]^\frac{1}{a},
	\end{equation}
	
	For any $r < s < t$ and $x_0 \in [\gamma, R]$,\,$\mathbb{Q}[\Hat{X}_s \in \cdot, s < \Hat{\tau}_{k+1}\mid \hat{\tau}_k^i + \varsigma_k^i = r, \xi_k^i = x_0] = \mathbb{Q}[\Upsilon_r(-x_0) + \tilde{B}_s - \tilde{B}_r \in \cdot, s < \Hat{\tau}_{k+1}]$. This is the law of a stopped Brownian motion starting at $\Upsilon_r(-x_0)$ at time $s - r$. Therefore,
	\begin{equation}\label{eq: CONTROL ON THE MIDDLE TERM FOUR}
		\begin{split}
			&\int_0^t\int_{S_t} \E^{\mathbb{Q}}\left[\left.\nnorm{\partial_x G_{t-s}(y,\Hat{X}_s)}^{a}\ind_{\{\hat{\tau}_k + \varsigma_k^i \le s < \hat{\tau}_{k+1}\}}\right|\hat{\tau}_k^i + \varsigma_k^i = r, \xi_k^i = x_0\right]^{\frac{1}{a}}\diff y \diff s\\
			&\qquad \qquad \qquad \qquad \qquad \qquad  = \int_r^t\int_{S_t} \left(\int_{-\infty}^0 \nnorm{\partial_x G_{t-s}(y,x)}^{a}G_{s-r}(z,x)\diff x\right)^{\frac{1}{a}} \diff y \diff s\\
			&\qquad \qquad \qquad \qquad \qquad \qquad  = \int_0^{\tilde{t}}\int_{S_t} \left(\int_{-\infty}^0 \nnorm{\partial_x G_{\tilde{t}-\tilde{s}}(y,x)}^{a}G_{\tilde{s}}(z,x)\diff x\right)^{\frac{1}{a}} \diff y \diff \tilde{s}
		\end{split}
	\end{equation}
	with $z = \Upsilon_r(-x_0),\, \tilde{t} = t - r$ and $\tilde{s} = s - r$. \eqref{eq: CONTROL ON THE MIDDLE TERM FOUR} is the exact same equation as in \eqref{eq: CONTROL ON THE MIDDLE TERM ONE}. Therefore, repeating the calculations in \eqref{eq: CONTROL ON THE MIDDLE TERM ONE} - \eqref{eq: CONTROL ON THE MIDDLE TERM THREE}, we have the bound
	\begin{equation}
		\begin{split}
			&\int_0^t\int_{S_t} \E^{\mathbb{Q}}\left[\left.\nnorm{\partial_x G_{t-s}(y,\Hat{X}_s)}^{a}\ind_{\{\hat{\tau}_k + \varsigma_k^i \le s < \hat{\tau}_{k+1}\}}\right|\hat{\tau}_k^i + \varsigma_k^i = r, \xi_k^i = x_0\right]^{\frac{1}{a}}\diff y \diff s \\
			&\hspace{7cm}\le C_a \int_{-S_t} \tilde{t}^{-\frac{1}{a}} \nnorm{z}^{\frac{1}{a}} \tilde{y}^{\frac{1}{a}}e^{-\frac{(\tilde{y} - \nnorm{z})^2}{4a\tilde{t}}} \diff \tilde{y}.
		\end{split}
	\end{equation}
	Recalling $S_t = \Upsilon_t(S)$ and $\Upsilon_r(-x_0) = z$, employing the change of variable $\tilde{y} = - \Upsilon_t(-\tilde{x})$ and the bounds in \eqref{eq:GENERAL BOUNDS ON THE SCALE TRANSFORM} which hold for all $\tilde{x}$, we have
	\begin{equation}\label{eq: CONTROL ON THE MIDDLE TERM FIVE}
		\begin{split}
			&\int_0^t\int_{S_t} \E^{\mathbb{Q}}\left[\left.\nnorm{\partial_x G_{t-s}(y,\Hat{X}_s)}^{a}\ind_{\{\hat{\tau}_k + \varsigma_k^i \le s < \hat{\tau}_{k+1}\}}\right|\hat{\tau}_k^i + \varsigma_k^i = r, \xi_k^i = x_0\right]^{\frac{1}{a}}\diff y \diff s \\
			& \hspace{7cm}\le C_{a,\sigma} \int_{0}^{\varepsilon} \tilde{t}^{-\frac{1}{a}} x_0^{\frac{1}{a}} \tilde{x}^{\frac{1}{a}}e^{-\frac{c(\gamma - \varepsilon )^2}{4a\tilde{t}}} \diff \tilde{x}\\
			& \hspace{7cm}\le C_{a,\sigma} R^{\frac{1}{a}}\varepsilon^{1 + \frac{1}{a}},
		\end{split}
	\end{equation}
	where the final line follows from the fact that for any fixed $z > 0$, the function $t \mapsto t^{-1}e^{-z^2/4t}$ is bounded by $4e^{-1}z^{-2}$ and $x_0 \in [\gamma,R]$. Inserting this upper bound into \eqref{eq: APPLICATION OF FUBINI AND HOLDER TO GET A PRELIMINARY UPPERBOUND ON IK}, we obtain
	\begin{equation}\label{eq: FINAL UPPERBOUND ON IK}
		I_k \le C_{a,p,q,\sigma}R^{\frac{1}{a}}\varepsilon^{1 + \frac{1}{a}} \prob[\tau_k^i \le T]^\frac{1}{a}.
	\end{equation}
	Therefore, inserting the upper bounds \eqref{eq:BOUND ON THE I TERM IN THE DECOMPOSITION OF THE } and \eqref{eq: FINAL UPPERBOUND ON IK} into \eqref{eq: DECOMPOSITION OF THE MIDDLE TERM INTO AN INFINITE SUM OF INTEGRALS}, we obtain the bound
	\begin{equation}\label{eq:BOUND ON THE SECOND TIME IN THE DECOMPOSITION THAT WE CARE ABOUT}
		\int_0^{t_0}\E\left[\nnorm{\hat{b}_s^i}\ind_{\{\Hat{X}_s < 0\}}\nnorm{\partial_x u(s,\Hat{X}_s)}\right] \diff s
		\le C \varepsilon^{1 + \frac{1}{a}}\left(t^{-\frac{1}{a}}\E[\nnorm{X_0}]^{\frac{1}{a}} + \sum_{k \ge 1} \prob[\tau_k^i \le T]^{\frac{1}{a}}\right).
	\end{equation}
	Inserting the bounds \eqref{eq:BOUND ON THE FIRST TIME IN THE DECOMPOSITION THAT WE CARE ABOUT}, \eqref{eq:BOUND ON THE THIRD TIME IN THE DECOMPOSITION THAT WE CARE ABOUT} and \eqref{eq:BOUND ON THE SECOND TIME IN THE DECOMPOSITION THAT WE CARE ABOUT} into \eqref{eq:THE DECOMPOSITION WE WISH TO CONTROL IN THE BOUNDARY ESTIMATES} and using \cref{prop: EXPONENTIAL DECAY OF THE HITTING PROBABILITIES} to control the sum of the hitting time probabilities, we obtain
	\begin{equation}
		\prob\left[X_t^i \in S,\, X_t^i < 0\right] = O(t^{-\frac{\delta}{2}}\varepsilon^{1 + \beta})
	\end{equation}
	for some $\delta \in (0,1],\, \beta > 0$ and all $\varepsilon \ll 1$. This completes the proof.
\end{proof}

Combining the results obtained above, we have the following.

\begin{corollary}[Regularity of the empirical measure]\label{cor: REGULARITY CONTROL ON THEM EMPIRICAL MEASURES}
	The empirical measures $\nu^N$ satisfy, uniformly in $N \ge 1$ and $t \in (0,T]$,
	\begin{equation}\label{eq: REFULARITY PROPERTIES OF THE EMPIRICAL MEASURES OF THE NEURO PARTICLE SYSTEM WITH JUMPS}
		\begin{cases}
			\exists \varepsilon > 0 \qquad &\E\left[\nu_t^N(-\infty,-a)\right] = O(e^{-\varepsilon a^2})\quad \text{as}\quad a \to \infty,\\
			\exists \delta \in (0,1] \qquad &\E\left[\nu_t^N(a,b)\right] \le Ct^{-\delta/2}\nnorm{b-a}^\delta \quad \text{for}\quad a < b < 0,\\
			\exists \delta \in (0,1], \beta > 0 \qquad &\E\left[\nu_t^N(-\varepsilon,0)\right] = t^{-\delta/2}O(\varepsilon^{1+\beta})\quad \text{as}\quad \varepsilon \to 0.
		\end{cases}
	\end{equation}
\end{corollary}
\begin{proof}
	By the sub-Gaussianity of $X_t^i$, which is uniform in $i$ and $N$, the first claim is a direct consequence of Corollary \ref{cor: SUBGUASSIANTY COROLLARY}. For the second claim, by \cref{prop: PROPOSTION SHOWING THAT WE HAVE AN UPPER BOUND ON THE PROBABILITYS OF XTI TAKING VALUES IN SOME SET}
	\begin{equation*}
		\E\left[\nu_t^N(a,b)\right] \le C t^{-\frac{\delta}{2}}\operatorname{Leb}((a,b))^{\delta}\frac{1}{N}\sum_{i=1}^N\left[1 + \sum_{k \ge 1}\prob\left[\tau_k^i \le t\right]^\delta\right],
	\end{equation*}
	for all $\delta > 0$, close enough to $0$, where $N^{-1}\sum_{i=1}^N\left[1 + \sum_{k \ge 1}\prob\left[\tau_k^i \le t\right]^\delta\right]$ is bounded independently of $N$ by \cref{prop: EXPONENTIAL DECAY OF THE HITTING PROBABILITIES}.
	The third claim is exactly the statement of \cref{prop: BOUNDARY DECAY OF THE EMPIRICAL MEAUSRES IN EXPECTATION}.
\end{proof}

\section{Control on the Increments of the Cumulative Spike Count}\label{proof: PROOF ON CONTROL OF THE FIRINGS}

We establish some regularity properties of the increments of the cumulative spike count and their second moments. This can serve to conclude that $F^{D,N}$ is tight and that the limiting process is continuous.

\begin{proposition}\label{prop: PROPOSITION THAT WE CAN CONTROL THE FIRST MOMENT OF THE INCREMENTS OF THE  DELAYED FIRING FUNCTION}
	For any $s,t \in [0,T]$, we may find a $C > 0$, independent of $s$, $t$ and $N$, such that 
	\begin{equation}
		\E\left[\nnnorm{F_t^{D,N} - F_s^{D,N}}\right] \le C\nnnorm{t- s}.
	\end{equation}
\end{proposition}
\begin{proof}
	Without loss of generality, we may suppose that $s < t$. As $\varsigma_k^i$ is independent of $\tau_k^i$ for any $i$ and $k$, then
	\begin{align}\label{eq: GENERAL BOUND THAT THE RESET TIME IS IN AN INTERVAL}
		\begin{split}
			\prob\left[\tau_k^i + \varsigma_k^i \in (s,t]\right] &= \int_0^t\int_{s-u}^{t-u}\pref(u) \diff\prob\left[\tau_k^i \in \diff v\right] \diff u\\
			&= \int_0^s\int_{s-v}^{t-v}\pref(u) \diff u \diff\prob\left[\tau_k^i \in \diff v\right]  + \int_s^t\int_{0}^{t-v}\pref(u) \diff u \diff\prob\left[\tau_k^i \in \diff v\right] \\
			&\le 2 \norm{\pref}_{\infty}(t - s)\prob\left[\tau_k^i \le T\right]. 
		\end{split}
	\end{align}
	Hence, by definition of $F^{D,N}$ in \eqref{eq:F-DN,J-Di},
	\begin{align*}
		\E\left[\nnnorm{F_t^{D,N} - F_s^{D,N}}\right] 
		&= \frac{1}{N} \sum_{i = 1}^N \sum_{k \ge 1} \prob\left[\tau_k^i + \varsigma_k^i \in (s,t]\right]\\
		&\le 2 \norm{\pref}_{\infty}(t - s) \frac{1}{N} \sum_{i = 1}^N \sum_{k \ge 1} \prob\left[\tau_k^i \le T\right].
	\end{align*}
	By \cref{prop: EXPONENTIAL DECAY OF THE HITTING PROBABILITIES}, $N^{-1}\sum_{i = 1}^N \sum_{k \ge 1} \prob\left[\tau_k^i \le T\right]$ is bounded uniformly in $N$.
\end{proof}

\begin{proposition}\label{prop: PROPOSITION THAT WE CAN CONTROL THE SECOND MOMENT OF THE INCREMENTS OF THE  DELAYED FIRING FUNCTION}
	For any $s,t \in [0,T]$, there exist a $\beta > 0$ and a $C = C(\beta) > 0$, independent of $s$, $t$ and $N$, such that 
	\begin{equation}
		\E\left[\nnnorm{F_t^{D,N} - F_s^{D,N}}^2\right] \le C(\nnnorm{t-s}N^{-1} + \nnnorm{t- s}^{1 + \beta}).
	\end{equation}
\end{proposition}
\begin{proof}
	Without loss of generality, we may suppose that $s \le t$. By definition of $F^{D,N}$, we have
	\begin{align*}
		\E\left[\nnnorm{F_t^{D,N} - F_s^{D,N}}^2\right] =& \frac{1}{N^2} \sum_{i = 1}^N\sum_{k \ge 1} \prob \left[\tau_k^i + \varsigma_k^i \in (s,t]\right]\\
		&+ \frac{1}{N^2}\sum_{i = 1}^N \sum_{k \neq \ell} \prob\left[\tau_k^i + \varsigma_k^i \in (s,t],\tau_\ell^i + \varsigma_\ell^i \in (s,t]\right]\\
		&+ \frac{1}{N^2}\sum_{\substack{i,j = 1 \\ i \ne j}}^N \sum_{k,\ell \ge 1} \prob\left[\tau_k^i + \varsigma_k^i \in (s,t],\tau_\ell^j + \varsigma_\ell^j \in (s,t]\right]\\
		&= I  + II + III.  
	\end{align*} 
	The aim is to analyse each term above separately. By \eqref{eq: GENERAL BOUND THAT THE RESET TIME IS IN AN INTERVAL}, the first term is upper bounded by
	\begin{equation*}
		I \le \frac{c(t-s)}{N^2} \sum_{i = 1}^N\sum_{k \ge 1} \prob\left[\tau_k^i \le T\right] \le \frac{C(t-s)}{N}.
	\end{equation*}
	The constant $C$ in the last inequality is independent of $N$ by \cref{prop: EXPONENTIAL DECAY OF THE HITTING PROBABILITIES}. For the second term, we observe that if $\ell > k$, then, by definition, $\tau_k^i + \varsigma_k^i \le \tau_\ell^i$ almost surely. As the system receives no information from $\varsigma_\ell^i$ until after time $\tau_\ell^i$, we have by construction that $\varsigma_\ell^i$ is independent of $\tau_k^i,\,\varsigma_k^i$, and $\tau_\ell^i$. Therefore,
	\begin{align*}
		&\prob\left[\tau_k^i + \varsigma_k^i \in (s,t],\tau_\ell^i + \varsigma_\ell^i \in (s,t]\right]\\
		&\qquad= \int_0^t\int_{s-u}^{t-u}\pref(u)\prob\left[\tau_k^i + \varsigma_k^i \in (s,t],v + u \in (s,t]\mid \tau_\ell^i = v \right] \diff\prob\left[\tau_\ell^i \in \diff v\right] \diff u\\
		&\qquad= \int_0^s\int_{s-v}^{t-v}\pref(u)\prob\left[\tau_k^i + \varsigma_k^i \in (s,t],v + u \in (s,t]\mid \tau_\ell^i = v \right] \diff\prob\left[\tau_\ell^i \in \diff v\right] \diff u\\
		&\qquad \quad + \int_s^t\int_{0}^{t-v} \pref(u)\prob\left[\tau_k^i + \varsigma_k^i \in (s,t],v + u \in (s,t]\mid \tau_\ell^i = v \right] \diff\prob\left[\tau_\ell^i \in \diff v\right] \diff u\\
		&\qquad\le 2 \norm{\pref}_{\infty}(t - s)\prob\left[\tau_k^i + \varsigma_k^i \in (s,t],\tau_\ell^i \le t\right]\\
		&\qquad\le 2 \norm{\pref}_{\infty}(t - s)\prob\left[\tau_k^i + \varsigma_k^i \in (s,t]\right]^{\frac{1}{p}}\prob\left[\tau_\ell^i \le t\right]^{\frac{1}{q}}\\
		&\qquad\le C(t - s)^{1 + \frac{1}{p}}\prob\left[\tau_k^i \le t\right]^{\frac{1}{p}}\prob\left[\tau_\ell^i \le t\right]^{\frac{1}{q}},
	\end{align*}
	where the penultimate line follows from H\"older's inequality and the last line follows from \eqref{eq: GENERAL BOUND THAT THE RESET TIME IS IN AN INTERVAL}. Therefore, by \cref{prop: EXPONENTIAL DECAY OF THE HITTING PROBABILITIES}, there is a larger $C > 0$ independent of $N$ but dependent on $p > 1$ such that $II \le C(t - s)^{1 + \frac{1}{p}}$. For the final term, as we have independence between the Brownian motions and the waiting times, on the event that $\tau_\ell^j$ has not occurred yet, $\varsigma_\ell^j$ must be independent of the other terms. That is, on the event $\{\tau_k^i + \varsigma_k^i \le \tau_\ell^j\}$, $\varsigma_\ell^j$ is independent of $\tau_k^i,\,\varsigma_k^i$ and $\tau_\ell^j$. Therefore,
	\begin{align*}
		&\prob\left[\tau_k^i + \varsigma_k^i \in (s,t],\tau_\ell^j + \varsigma_\ell^j \in (s,t], \tau_k^i \le \tau_\ell^j\right]\\
		&\qquad= \int_0^t\int_{s-u}^{t-u}\pref(u)\prob\left[\tau_k^i + \varsigma_k^i \in (s,t],v + u \in (s,t], \tau_k^i \le v \mid \tau_\ell^j = v \right] \diff\prob\left[\tau_\ell^j \in \diff v\right] \diff u\\ 
		&\qquad\le 2 \norm{\pref}_{\infty}(t - s)\prob\left[\tau_k^i + \varsigma_k^i \in (s,t],\tau_\ell^j \le t\right]\\
		&\qquad\le C(t - s)^{1 + \frac{1}{p}}\prob\left[\tau_k^i \le t\right]^{\frac{1}{p}}\prob\left[\tau_\ell^j \le t\right]^{\frac{1}{q}}.
	\end{align*}
	Therefore, by \cref{prop: EXPONENTIAL DECAY OF THE HITTING PROBABILITIES}, there is a larger $C > 0$ independent of $N$ but dependent on $p > 1$ such that $III \le C(t - s)^{1 + \frac{1}{p}}$.
\end{proof}

\begin{proposition}\label{prop: CONTROL ON THE PROPORTION OF NEURONS WHICH FIRE TWICE}
	For any given $T > 0$, consider $0 \le t \le t+h \le T$ with $h \in (0,\,1)$. Then, we may find a $C_p > 0$, independent of $h$, and an integer $N_0 = N_0(h)$ such that for all $N \ge N_0$, we have
	\begin{equation}
		\prob\left[F_{t+h}^N - F_{t-}^N > 1 + C_p h^{1/16}\right] \le C_p h^p,
	\end{equation}
	for any $p \ge 1$.
\end{proposition}

\begin{proof}
	\noindent \underline{Step 1:}\\[1ex]
	We will first show that the proportion of particles that spike twice in a small interval tends to $0$ with the length of the interval uniformly in $N$ for $N$ large. To be more precise, given an interval $[t,t+h]$ and $\beta \in (0,1)$ we define
	\begin{equation*}
		\tau(\beta) \coloneqq \inf \{s \in [t,t+h]\,:\, \frac{1}{N}\sum_{i=1}^N \ind_{\{J_s^i - J_{t-}^i\ge 2\}}\ge \beta\}, \qquad \inf \emptyset = +\infty.
	\end{equation*}
	We shall show for $\beta = h^{1/4}$, $N > h^{-1/2}$ and $p \ge 1$ there exists $C_p$ such that $\prob\left[\tau(\beta)  \le t+ h \right] \le C_ph^p$. 
	
	Now let $\mathcal{I}^{(2)} = \{i \in \{1,\ldots,N\}\,:\, J_{t+h}^i - J_{t-}^i \ge 2\}$, i.e. the random set of particles that spike at least twice on $[t,t+h]$. For the rest of this step, we will be working on the event $\{\tau(\beta) \le t+h\}$. By right-continuity of $J^i$, we have $\nnnorm{\mathcal{I}^{(2)}} \ge N\beta$. Choose $\mathcal{I}^{(2)}(\beta) \subset \mathcal{I}^{(2)}$ such that $N\beta \le \nnnorm{\mathcal{I}^{(2)}(\beta)} < N\beta + 1$, e.g. $\mathcal{I}^{(2)}(\beta)$  is the first $\lceil N\beta \rceil$ particles to spike twice.
	
	\sloppy{
		Now for any $i \in \mathcal{I}^{(2)}(\beta)$, when $X^i$ spikes for the second time, the particle $Z^i$ must have had $\sup_{s \le t} Z_s^i \le \sum_{\tilde{k} = 1}^{k-1}\xi_{\tilde{k}}^i$ and $\sup_{s \le t + h} Z_s^i \ge \sum_{\tilde{k} = 1}^{k}\xi_{\tilde{k}}^i$ for some $k$ (stochastic). Furthermore, as this is the second spike in $[t,\,t+h]$, the value of $Z^i$ must have increased by at least $\gamma$ compared with its value just before time $t$. Hence, $\gamma \le \sup_{t \le s \le t + h} \nnnorm{Z_s^i - Z_{t-}^i}$. Therefore,}
	\begin{align*}
		\gamma &\le \sup_{t \le s \le t + h} \nnnorm{Z_s^i - Z_{t-}^i}\\
		&\le \int_t^{t+h} \nnnorm{b(s,X_s^i,\nu_s^N,\mathfrak{f}_s^N)}\diff{}s +  \sup_{t \le s \le t + h}\nnnorm{\int_t^s\sigma(u,X_u^i)\diff{}B_u^i}\\
		&\le Ch\left[1 + \sup_{s \le t + h} \nnnorm{Z_s^i} + \frac{1}{N}\sum_{j = 1}^N \; \sup_{s \le t + h} \nnnorm{Z_s^j}\right] +  \sup_{t \le s \le t + h}\nnnorm{\int_t^s\sigma(u,X_u^i)\diff{}B_u^i},
	\end{align*}
	where the last inequality is due to the linear growth condition on $b$, see \eqref{eq: UPPER BOUND ON B IN TERMS OF Z AND ITS EMPIRICAL AVERAGE}. Now summing over the particles in $\mathcal{I}^{(2)}(\beta)$ and dividing by $N$, we obtain 
	\begin{align*}
		\frac{\gamma\nnnorm{\mathcal{I}^{(2)}(\beta)}}{N} &\le
		\frac{1}{N} \sum_{i \in \mathcal{I}^{(2)}(\beta)}Ch\left[1 + \sup_{s \le t + h} \nnnorm{Z_s^i} + \frac{1}{N}\sum_{j = 1}^N \; \sup_{s \le t + h} \nnnorm{Z_s^j}\right]\\
		&\qquad + \frac{1}{N}\sum_{i \in \mathcal{I}^{(2)}(\beta)} \sup_{t \le s \le t + h}\nnnorm{\int_t^s\sigma(u,X_u^i)\diff{}B_u^i}.
	\end{align*}
	By using $N\beta \le \nnnorm{\mathcal{I}^{(2)}(\beta)} < N\beta + 1$ and the Cauchy-Schwarz inequality we have
	\begin{align*}
		\gamma\beta &\le
		Ch \left(\frac{\nnnorm{\mathcal{I}^{(2)}(\beta)}}{N}\right)^{1/2}\left(\frac{1}{N} \sum_{i = 1}^N\left[1 + \sup_{s \le t + h} \nnnorm{Z_s^i} + \frac{1}{N}\sum_{j = 1}^N \; \sup_{s \le t + h} \nnnorm{Z_s^j}\right]^2\right)^{1/2} \\
		&\qquad+ \left(\frac{\nnnorm{\mathcal{I}^{(2)}(\beta)}}{N}\right)^{1/2} \left(\frac{1}{N}\sum_{i = 1}^N\sup_{t \le s \le t + h}\nnnorm{\int_t^s\sigma(u,X_u^i)\diff{}B_u^i}^2\right)^{1/2}\\
		&\le 
		Ch \left(\beta + \frac{1}{N}\right)^{1/2}\left(\frac{1}{N} \sum_{i = 1}^N\left[1 + \sup_{s \le t + h} \nnnorm{Z_s^i} + \frac{1}{N}\sum_{j = 1}^N \; \sup_{s \le t + h} \nnnorm{Z_s^j}\right]^2\right)^{1/2} \\
		&\qquad+ \left(\beta + \frac{1}{N}\right)^{1/2} \left(\frac{1}{N}\sum_{i = 1}^N\sup_{t \le s \le t + h}\nnnorm{\int_t^s\sigma(u,X_u^i)\diff{}B_u^i}^2\right)^{1/2}.
	\end{align*}
	As $\beta = h^{1/4} \ge N^{-1/2}$, we have $1/(\beta N) \le N^{-1/2} \le 1$. Therefore, dividing the above by $\beta$, we obtain
	\begin{align*}
		\gamma 
		\le& Ch\beta^{-1/2}\left(\frac{1}{N} \sum_{i = 1}^N\left[1 + \sup_{s \le t + h} \nnnorm{Z_s^i} + \frac{1}{N}\sum_{j = 1}^N \; \sup_{s \le t + h} \nnnorm{Z_s^j}\right]^2\right)^{1/2}\\
		&+ 2 \beta^{-1/2}\left(\frac{1}{N}\sum_{i = 1}^N\sup_{t \le s \le t + h}\nnnorm{\int_t^s\sigma(u,X_u^i)\diff{}B_u^i}^2\right)^{1/2}\\
		\le& Ch\beta^{-1/2}\left(1 + \frac{1}{N}\sum_{i = 1}^N \; \sup \limits_{s \le t + h} \nnnorm{Z_s^i}^2\right)^{1/2} \\
		&+ 2 \beta^{-1/2}\left(\frac{1}{N}\sum_{i = 1}^N\sup_{t \le s \le t + h}\nnnorm{\int_t^s\sigma(u,X_u^i)\diff{}B_u^i}^2\right)^{1/2},
	\end{align*}
	where $C$ becomes larger between inequalities. Hence, by Markov's inequality, we have shown
	\begin{align}
		\begin{split}
			\prob\left[\tau(\beta) \le t+h \right] &\le C^ph^p\gamma^{-p}\beta^{-p/2}\E\left[\left(1 + \frac{1}{N}\sum_{i = 1}^N \; \sup \limits_{s \le t + h} \nnnorm{Z_s^i}^2\right)^{p/2} \right] \\
			&+ 2^p\gamma^{-p}\beta^{-p/2}\E\left[\left(\frac{1}{N}\sum_{i = 1}^N\sup_{t \le s \le t + h}\nnnorm{\int_t^s\sigma(u,X_u^i)\diff{}B_u^i}^2\right)^{p/2}\right].
		\end{split}
	\end{align}
	For $p \le 2$, by Jensen's inequality and \cref{prop: BOUNDS ON THE SUP OF THE P TH MOMENTS OF THE SUP NORM OF Z},
	\begin{equation*}
		\E\left[\left(1 + \frac{1}{N}\sum_{i = 1}^N \; \sup \limits_{s \le t + h} \nnnorm{Z_s^i}^2\right)^{p/2} \right] \le \left(\E\left[1 + \frac{1}{N}\sum_{i = 1}^N \; \sup \limits_{s \le t + h} \nnnorm{Z_s^i}^2\right] \right)^{p/2} \le C.
	\end{equation*}
	By Jensen's inequality and the Burkholder-Davis-Gundy inequality,
	\begin{align*}
		&\E\left[\left(\frac{1}{N}\sum_{i = 1}^N\sup_{t \le s \le t + h}\nnnorm{\int_t^s\sigma(u,X_u^i)\diff{}B_u^i}^2\right)^{p/2}\right] \\
		&\hspace{2cm}\le \left(\E\left[\frac{1}{N}\sum_{i = 1}^N\sup_{t \le s \le t + h}\nnnorm{\int_t^s\sigma(u,X_u^i)\diff{}B_u^i}^2\right]\right)^{p/2} \le Ch^{p/2}.
	\end{align*}
	Similarly, for $p > 2$, by Jensen's inequality,
	\begin{equation*}
		\E\left[\left(1 + \frac{1}{N}\sum_{i = 1}^N \; \sup \limits_{s \le t + h} \nnnorm{Z_s^i}^2\right)^{p/2} \right] \le C_p \E\left[1 + \frac{1}{N}\sum_{i = 1}^N \; \sup \limits_{s \le t + h} \nnnorm{Z_s^i}^{p}\right] \le C.
	\end{equation*}
	By Jensen's and the Burkholder-Davis-Gundy inequality,
	\begin{align*}
		&\E\left[\left(\frac{1}{N}\sum_{i = 1}^N\sup_{t \le s \le t + h}\nnnorm{\int_t^s\sigma(u,X_u^i)\diff{}B_u^i}^2\right)^{p/2}\right] \\
		&\hspace{2cm}\le \E\left[\frac{1}{N}\sum_{i = 1}^N\sup_{t \le s \le t + h}\nnnorm{\int_t^s\sigma(u,X_u^i)\diff{}B_u^i}^p\right] \le Ch^{p/2}.
	\end{align*}
	Therefore,
	\begin{equation*}
		\prob\left[\tau(\beta) \le t+h \right] \le C_p(h^p\beta^{-p/2} + h^{p/2}\beta^{-p/2}) = C_p(h^{7p/8} + h^{3p/8}).
	\end{equation*}
	Now recall $\mathcal{I}^{(2)} = \{i \in \{1,\ldots,N\}\,:\, J_{t+h}^i - J_{t-}^i \ge 2\}$. On the event $\{\tau(\beta) > t + h\} \cap \{N^{-1}\sum_{i=1}^N(J_{t+h}^i)^2 \le h^{-1/8}\}$ we have $\nnnorm{\mathcal{I}^{(2)}} < N\beta$, therefore
	\begin{align*}
		F_{t+h}^N - F_{t-}^N &= \frac{1}{N}\sum_{i = 1}^N (J_{t+h}^i - J_{t-}^i) = \frac{1}{N}\sum_{i \in (\mathcal{I}^{(2)})^\complement} (J_{t+h}^i - J_{t-}^i) +  \frac{1}{N}\sum_{i \in \mathcal{I}^{(2)}} (J_{t+h}^i - J_{t-}^i) \\
		&\le 1 + \left(\frac{\nnnorm{\mathcal{I}^{(2)}}}{N}\right)^{1/2}\left(\frac{1}{N}\sum_{i = 1}^N \left(J_{t+h}^i\right)^2\right)^{1/2}\\
		&\le 1 + \beta^{1/2} \cdot h^{-1/16} = 1 + h^{1/16}.
	\end{align*}
	Hence,
	\begin{align*}
		\prob\left[F_{t+h}^N - F_{t-}^N > 1 + h^{1/16}\right] &\le \prob\left[\tau(\beta) \le t+h \right] + \prob\left[N^{-1}\sum_{i=1}^N(J_{t+h}^i)^2 > h^{-1/8}\right]\\
		&\le C_p(h^{7p/8} + h^{3p/8}+ h^{p/8}).
	\end{align*}
	By first replacing $p$ with $8p$ in the above and then employing the fact that $h \in (0,1)$, the result follows. 
\end{proof}

We also observe that the probability of seeing a large proportion of neurons firing is small for all sufficiently large $N$, provided that the time domain we are considering, $[t,t+h]$, is small enough.

\begin{proposition}\label{prop: PROPOSITION THAT WE CAN CONTROL THE INCREMENTS OF THE FIRING FUNCTION}
	For every $t \in [0,\,T]$ and $\eta > 0$, we have
	\begin{equation*}
		\lim_{h \to  0} \limsup_{N \to  +\infty} \prob\left[F_{(t+h) \wedge T}^N - F_t^N \ge \eta\right] = 0 \quad \textnormal{and} \quad \lim_{h \to  0} \limsup_{N \to  +\infty} \prob\left[F_{t}^N - F_{(t-h) \vee 0}^N \ge \eta\right] = 0
	\end{equation*}
\end{proposition}

\begin{proof}
	
	By Proposition \ref{prop: CONTROL ON THE PROPORTION OF NEURONS WHICH FIRE TWICE}, the probability that a large proportion of neurons fire at least twice in a small time interval is small, uniformly for all $N$ large. Hence, we would like to gain some control over the number of neurons that fire provided a small proportion of them fire twice. As the $p$-th moments of the number of firings are uniformly bounded in $N$, we may exploit this using Markov's inequality to control the contribution to $F^N$ by the neurons that fire twice or more. Lastly, to control the neurons that fire at most once, we employ the techniques in \citep[Appendix~A.2]{hambly2019spde}.\\
	
	\noindent \underline{Step 1:}\\[1 ex]
	Choose $h < 1,\, t \in [0,T]$, let $N > h^{-1/2}$ and $\beta = h^{1/4}$. As in Proposition \ref{prop: CONTROL ON THE PROPORTION OF NEURONS WHICH FIRE TWICE}, we define the set $\mathcal{I}^{(2)} = \{i \in \{1,\ldots,N\}\,:\, J_{t+h}^i - J_{t-}^i \ge 2\}$, i.e. the set of particles which jump twice on the interval $[t,t+h]$. Choose a deterministic set $\mathcal{I}_0 \subset \{1,\ldots,N\}$. On the event $A(\mathcal{I}_0)\coloneqq \{\mathcal{I}^{(2)} = \mathcal{I}_0\} \cap \{\nnnorm{\mathcal{I}^{(2)}} < N\beta\}$, we have
	\begin{align*}
		F_{t+h}^N - F_t^N &= \frac{1}{N} \sum_{i = 1}^N J_{t+h}^i - J_{t}^i  = \frac{1}{N} \sum_{i \in \mathcal{I}_0} J_{t+h}^i - J_{t}^i + \frac{1}{N} \sum_{i \in \mathcal{I}_0^\complement} J_{t+h}^i - J_{t}^i\\
		&\le \left(\frac{\nnnorm{\mathcal{I}_0}}{N}\right)^{1/2}\left(\frac{1}{N}\sum_{i \in \mathcal{I}_0} \left(J_{t+h}^i\right)^2\right)^{1/2}
		+ \frac{1}{N} \sum_{i \in \mathcal{I}_0^\complement} \left(J_{t+h}^i - J_{t}^i\right)\ind_{\{J_{t+h}^i - J_{t-}^i \le 1\}}\\
		&\le \beta^{1/2}\left(\frac{1}{N}\sum_{i = 1}^N \left(J_{t+h}^i\right)^2\right)^{1/2}
		+ \frac{1}{N} \sum_{i = 1}^N \left(J_{t+h}^i - J_{t}^i\right)\ind_{\{J_{t+h}^i - J_{t}^i \le 1\}}\\
	\end{align*}
	where we used the fact that $J_{t+h}^i - J_{t}^i \le J_{t+h}^i - J_{t-}^i$ to drop the dependence on $t-$ in the indicator function in the last inequality. Therefore, by the above and the law of total probability
	{\small \begin{align*}
			&\prob\left[F_{t+h}^N - F_t^N \ge 2\eta, \nnnorm{\mathcal{I}^{(2)}} < N\beta\right]\\
			&= \sum_{\mathcal{I}_0 \subset \{1,\ldots,N\}} \prob\left[F_{t+h}^N - F_t^N \ge 2\eta, A(\mathcal{I}_0) \right]\\
			&\le \sum_{\mathcal{I}_0 \subset \{1,\ldots,N\}} \prob\left[\beta^{1/2}\left(\frac{1}{N}\sum_{i = 1}^N \left(J_{t+h}^i\right)^2\right)^{1/2}
			+ \frac{1}{N} \sum_{i = 1}^N \left(J_{t+h}^i - J_{t}^i\right)\ind_{\{J_{t+h}^i - J_{t}^i \le 1\}} \ge 2\eta, A(\mathcal{I}_0) \right]\\
			&\le \sum_{\mathcal{I}_0 \subset \{1,\ldots,N\}} \prob\left[\beta^{1/2}\left(\frac{1}{N}\sum_{i = 1}^N \left(J_{t+h}^i\right)^2\right)^{1/2} \ge \eta,A(\mathcal{I}_0) \right] \\
			&\hspace{1cm}+\sum_{\mathcal{I}_0 \subset \{1,\ldots,N\}} \prob\left[\frac{1}{N} \sum_{i = 1}^N \left(J_{t+h}^i - J_{t}^i\right)\ind_{\{J_{t+h}^i - J_{t}^i \le 1\}} \ge \eta, A(\mathcal{I}_0) \right]\\
			&\le \frac{\beta^{1/2}}{\eta}\E\left[\left(\frac{1}{N}\sum_{i = 1}^N \left(J_{t+h}^i\right)^2\right)^{1/2}\right]
			+\prob\left[\frac{1}{N} \sum_{i = 1}^N \left(J_{t+h}^i - J_{t}^i\right)\ind_{\{J_{t+h}^i - J_{t}^i \le 1\}} \ge \eta \right]
	\end{align*}}
	where the first term is bounded uniformly in $N$ by Jensen's inequality and \cref{prop: BOUNDS ON THE SUP OF THE P TH MOMENTS OF THE SUP NORM OF Z}.\\
	
	\noindent\underline{Step 2:}\\[1 ex]
	We now look to analyse $\prob\left[\frac{1}{N} \sum_{i = 1}^N \left(J_{t+h}^i - J_{t}^i\right)\ind_{\{J_{t+h}^i - J_{t}^i \le 1\}} \ge \eta \right]$ obtained in step 1. The rest of this proof shall follow similarly as in \citep[Appendix~A.2]{hambly2019spde}. By fixing an $\varepsilon = \varepsilon(h)$ and $a = a(h)$, to be specified later, such that $\varepsilon(h) \to  0$ and $a(h)/h \to  \infty$ as $h \to  0$, then by Markov's inequality, Corollary \ref{cor: REGULARITY CONTROL ON THEM EMPIRICAL MEASURES}, and the uniform in $N$ sub-Gaussianity of the particles, we have
	\begin{align*}
		\prob\left[\nu_t^N(-\varepsilon,\,0) \ge \eta/4\right] &= o(1) \qquad \text{as } h \to  0, \\
		\prob\left[\frac{1}{N}\sum_{i = 1}^N \ind_{\sup_{s \le t + h} \nnnorm{Z_{s}^i} \ge \frac{a}{2h}} \ge \eta/4\right] &= o(1) \qquad \text{as } h \to  0,
	\end{align*}
	uniformly in $N$. Hence, we now look to control
	{\small \begin{equation}\label{eq: THE MAIN OBJECT OR PROBABILITY WE LOOK TO HAVE A CONTROL ON IN H IN THE PROOF THAT WE HAVE THE FIRING FUNCTION BEING TIGHT}
			\prob\left[\frac{1}{N} \sum_{i = 1}^N \left(J_{t+h}^i - J_{t}^i\right)\ind_{\{J_{t+h}^i - J_{t}^i \le 1\}} \ge \eta,\, \nu_t^N(-\varepsilon,0) < \eta/4,  \frac{1}{N}\sum_{i = 1}^N \ind_{\sup_{s \le t +  h} \nnnorm{Z_{s}^i} \ge \frac{a}{2h}} < \eta/4 \right].
	\end{equation}}
	First, we let $E$ denote the event in \eqref{eq: THE MAIN OBJECT OR PROBABILITY WE LOOK TO HAVE A CONTROL ON IN H IN THE PROOF THAT WE HAVE THE FIRING FUNCTION BEING TIGHT}, and we define two sets:
	\begin{align*}
		\mathcal{I} &= \left\{i \in \{1,\ldots,N\}\,;\, X_t^{i} \le -\varepsilon,\sup_{s \le t + h} \nnnorm{Z_{s}^{i}} < \frac{a}{2h}\right\} \qquad \textnormal{and} \\
		\mathcal{I}^{{\textnormal{ref}}} &= \left\{i \in \{1,\ldots,N\}\,;\, X_t^{i} = 0,\sup_{s \le t + h} \nnnorm{Z_{s}^{i}} < \frac{a}{2h}\right\}.
	\end{align*}
	$\mathcal{I}$ is the set of indices of the particles that are not too close to the boundary but also not too far away. $\mathcal{I}^{{\textnormal{ref}}}$ represents the set of indices of the particles that are in their refractory period at time $t$. We observe that by the countable additivity of probability measures
	\begin{equation}\label{eq: THE MAIN OBJECT OR PROBABILITY WE LOOK TO HAVE A CONTROL ON IN H IN THE PROOF THAT WE HAVE THE FIRING FUNCTION BEING TIGHT 1}
		\prob\left[E\right] \le \sum_{\substack{\mathcal{I}_0,\mathcal{I}_0^{{\textnormal{ref}}}\subset \{1,\ldots,N\} \\ \mathcal{I}_0 \cap \mathcal{I}_0^{\textnormal{ref}} = \emptyset}}\prob\left[E \cap \{\mathcal{I} = \mathcal{I}_0\} \cap \{\mathcal{I}^{\textnormal{ref}} = \mathcal{I}_0^{\textnormal{ref}}\} \right].
	\end{equation}
	We observe that $E \cap \{\mathcal{I} = \mathcal{I}_0\} \cap \{\mathcal{I}^{\textnormal{ref}} = \mathcal{I}_0^{\textnormal{ref}}\}$ is a subset of the event
	{\small\begin{equation}\label{eq: THE MAIN OBJECT OR PROBABILITY WE LOOK TO HAVE A CONTROL ON IN H IN THE PROOF THAT WE HAVE THE FIRING FUNCTION BEING TIGHT 2}
			\begin{aligned}
				&E \cap \{\mathcal{I} = \mathcal{I}_0\} \cap \{\mathcal{I}^{\textnormal{ref}} = \mathcal{I}_0^{\textnormal{ref}}\} \cap \Bigr\{\#\{i \in \mathcal{I}_0\,:\,\sup_{t\le s\le t + h}X_s^i \ge 0,X_t^i\le -\varepsilon\} \ge N\eta/4\Bigl\}\\
				&\bigcup E \cap \{\mathcal{I} = \mathcal{I}_0\} \cap \{\mathcal{I}^{\textnormal{ref}} = \mathcal{I}_0^{\textnormal{ref}}\} \cap \Bigr\{\#\{i \in \mathcal{I}_0^{\textnormal{ref}}\,:\,\sup_{\hat{\tau}^i\le s\le t + h}X_s^i \ge 0, \hat{\tau}^i \le t + h\} \ge N\eta/4\Bigl\},\\
			\end{aligned}
	\end{equation}}
	where $\hat{\tau}^i$ is the time that particle $i$ leaves its refractory period, i.e.
	\begin{equation*}
		\hat{\tau}^i \coloneqq \inf\{s \in [t,t+h]\,;\, X_s^i < 0 \} \qquad \textnormal{with} \qquad \inf\emptyset = \infty.
	\end{equation*}
	
	Therefore, replacing the event \eqref{eq: THE MAIN OBJECT OR PROBABILITY WE LOOK TO HAVE A CONTROL ON IN H IN THE PROOF THAT WE HAVE THE FIRING FUNCTION BEING TIGHT 2} into \eqref{eq: THE MAIN OBJECT OR PROBABILITY WE LOOK TO HAVE A CONTROL ON IN H IN THE PROOF THAT WE HAVE THE FIRING FUNCTION BEING TIGHT 1} and employing the subadditivity of the probability measures, we deduce
	\begin{align*}
		\prob\left[E\right] &\le \sum_{{\mathcal{I}_0 \subset \{1,\ldots,N\}}} \prob\left[\left.\#\{i \in \mathcal{I}_0\,:\,\sup_{t\le s\le t + h}X_s^i \ge 0,X_t^i\le -\varepsilon\} \ge N\eta/4 \right|\mathcal{I} = \mathcal{I}_0\right]\\
		&\hspace{2.5cm}\times\prob\left[\mathcal{I} = \mathcal{I}_0\right]\\
		&+ \sum_{{\mathcal{I}_0^{\textnormal{ref}} \subset \{1,\ldots,N\}}}\prob\left[\left.\#\{i \in \mathcal{I}_0^{\textnormal{ref}}\,:\,\sup_{\hat{\tau}^i\le s\le t + h}X_s^i \ge 0, \hat{\tau}^i \le t + h\} \ge N\eta/4 \right|\mathcal{I}^{\textnormal{ref}} = \mathcal{I}_0^{\textnormal{ref}}\right]\\
		&\hspace{2.5cm}\times\prob\left[\mathcal{I}^{\textnormal{ref}} = \mathcal{I}_0^{\textnormal{ref}}\right].
	\end{align*}
	Now, recall that particles in their refractory period are reset to at least $\gamma$ away from the boundary at zero by \cref{ass: ASSUMPTIONS FROM SOJMARK SPDE PAPER APPLIED TO THE NEUROSCIENCE SETTING} \eqref{ass: ASSUMPTIONS FROM SOJMARK SPDE PAPER APPLIED TO THE NEUROSCIENCE SETTING FOUR}. Since $\varepsilon$ will tend to zero, we can assume $\varepsilon < \gamma$, and we can then deduce that
	\begin{equation}\label{eq: THE MAIN OBJECT OR PROBABILITY WE LOOK TO HAVE A CONTROL ON IN H IN THE PROOF THAT WE HAVE THE FIRING FUNCTION BEING TIGHT 3}
		\begin{aligned}
			\prob\left[E\right] &
			\le \sum_{\substack{\mathcal{I}_0 \subset \{1,\ldots,N\}}} \prob\left[\left.\#\{i \in \mathcal{I}_0\,:\,\sup_{t\le s\le t + h}\left\{X_s^{i} - X_t^{i}\right\} \ge \varepsilon\} \ge N\eta/4 \right|\mathcal{I} = \mathcal{I}_0\right]\\
			&\hspace{2.5cm}\times\prob\left[\mathcal{I} = \mathcal{I}_0\right]\\
			&+ \sum_{{\mathcal{I}_0^{\textnormal{ref}} \subset \{1,\ldots,N\}}} \prob\left[\left.\#\{i \in \mathcal{I}_0^{\textnormal{ref}}\,:\,\sup_{\hat{\tau}^i\le s\le t + h}\left\{X_s^{i} - X_{\hat{\tau}^i}^{i}\right\} \ge \varepsilon\} \ge N\eta/4 \right|\mathcal{I}^{\textnormal{ref}} = \mathcal{I}_0^{\textnormal{ref}}\right]\\
			&\hspace{2.5cm}\times\prob\left[\mathcal{I}^{\textnormal{ref}} = \mathcal{I}_0^{\textnormal{ref}}\right].
		\end{aligned}
	\end{equation}
	Now we will look to control each of the terms in \eqref{eq: THE MAIN OBJECT OR PROBABILITY WE LOOK TO HAVE A CONTROL ON IN H IN THE PROOF THAT WE HAVE THE FIRING FUNCTION BEING TIGHT 3} individually. First turning our attention to the first term on the R.H.S. of \eqref{eq: THE MAIN OBJECT OR PROBABILITY WE LOOK TO HAVE A CONTROL ON IN H IN THE PROOF THAT WE HAVE THE FIRING FUNCTION BEING TIGHT 3}, we shall apply the scale transform as in \cref{prop: SCALE TRANSFORM LEMMA } to remove the general diffusive term in front of the Brownian motion. We introduce $U_s^i := \Upsilon_{t+s}(X_{t+s}^i) - \Upsilon_{t+s}(X_t^i)$ and note that, arguing as in \cref{prop: SCALE TRANSFORM LEMMA }, we get
	\begin{align*}
		\diff{}U_s^i &= u_s^i\diff{s} + \ind_{\{X_{t+s}^i < 0\}}\sqrt{1 - \rho_{t+s}^2} \diff{}W_{t+s}^i + \ind_{\{X_{t+s}^i < 0\}}\rho_{t+s}\diff{}W_{t+s}^0\\
		&\qquad + \diff\sum_{k \ge 1} \Upsilon_{\tau_k^i + \varsigma_k^i}(-\xi_k^i)\ind_{(t,t+s]}(\tau_k^i + \varsigma_k^i) \\
		&\eqqcolon u_s^i\diff{s} + \diff{\tilde{I}_s^i}  + \diff{I_s} + \diff\sum_{k \ge 1} \Upsilon_{\tau_k^i + \varsigma_k^i}(-\xi_k^i)\ind_{(t,t+s]}(\tau_k^i + \varsigma_k^i)
	\end{align*}
	where $\rho_{t + s} = \rho({t+s},\nu_{t+s}^N,\mathfrak{f}_{t+s}^N)$ and the drift satisfies $\nnnorm{u_s^i} \le c_1(1 + \hat{\Lambda}_{t+s}^i)$. By the bound on $\sigma$, we have that
	\begin{equation*}
		X_{t+s}^i - X_{t}^i \ge \varepsilon \implies U_s^i = \int_{X_t^i}^{X^i_{t+s}}\sigma(t+s,y)^{-1}\diff y \ge \varepsilon C_\sigma^{-1} \ge c_2\varepsilon,
	\end{equation*}
	where the last inequality holds for any $C_\sigma^{-1} \ge c_2 > 0$. Therefore,
	\begin{equation*}
		\sup_{t\le s\le t + h}\left\{X_s^i - X_t^i\right\} \ge \varepsilon \implies \sup_{s\le h} U_s^i \ge c_2\varepsilon.
	\end{equation*}
	It follows that we can bound the first term on the R.H.S. of \eqref{eq: THE MAIN OBJECT OR PROBABILITY WE LOOK TO HAVE A CONTROL ON IN H IN THE PROOF THAT WE HAVE THE FIRING FUNCTION BEING TIGHT 3} by
	\begin{equation*}
		\sum_{\substack{\mathcal{I}_0 \subset \{1,\ldots,N\}}} \prob\left[\left.\#\{i \in \mathcal{I}_0\,:\,\sup_{s\le h}\left\{U_s^i\right\} \ge c_2\varepsilon\} \ge N\eta/4 \right|\mathcal{I} = \mathcal{I}_0\right] \prob\left[\mathcal{I} = \mathcal{I}_0\right]
	\end{equation*}
	Using the decomposition of $U_s^i$, the fact that $-\sum_{k \ge 1} \Upsilon_{\tau_k^i + \varsigma_k^i}(-\xi_k^i)\ind_{(t,t+s]}(\tau_k^i + \varsigma_k^i) \ge 0$, and the growth estimate on the drift $u_s^i$, 
	\begin{align*}
		\tilde{I}_s^i &= U_s^i - \int_0^s u_v^i \diff{v} - I_s -\sum_{k \ge 1} \Upsilon_{\tau_k^i + \varsigma_k^i}(-\xi_k^i)\ind_{(t,t+s]}(\tau_k^i + \varsigma_k^i)
		\ge U_s^i - h \sup_{s \le h} \nnnorm{u_s^i} - \sup_{s \le h} \nnnorm{I_s}\\
		&\ge U_s^i - c_1h(1 + \hat{\Lambda}_{t + h}^{i}) - \sup_{s \le h} \nnnorm{I_s}
		\ge U_s^i - c_1h(1 + \sup_{s \le t + h} \nnnorm{Z_{s}^i} + a/2h) - a
	\end{align*}
	on the event
	\begin{equation*}
		\left\{\frac{1}{N}\sum_{j = 1}^N \sup_{s \le t + h} \nnnorm{Z_{s}^j} < \frac{a}{2h}\right\} \cap \left\{\sup_{s \le h}\nnnorm{I_s} < a\right\}.
	\end{equation*}
	Hence, on the same event, if $i \in \mathcal{I}_0$ such that $\sup_{s\le h} U_s^i \ge c_2\varepsilon$ then
	\begin{equation*}
		\sup_{s\le h} \tilde{I}_s^i \ge \sup_{s\le h} U_s^i - c_1h(1 + a/h) - a \ge c_2\varepsilon - c_1h(1 + a/h) - a = c_2\varepsilon - c_1h - a(1 + c_1).
	\end{equation*}
	Therefore, by splitting up the probability on this event and its complement and using the fact that the largest set $\mathcal{I}_0$ may be is $\{1,\,\ldots,\,N\}$, it follows that we can bound the first term on the R.H.S. of \eqref{eq: THE MAIN OBJECT OR PROBABILITY WE LOOK TO HAVE A CONTROL ON IN H IN THE PROOF THAT WE HAVE THE FIRING FUNCTION BEING TIGHT 3} by
	\begin{align*}
		&\prob\left[\#\{i \le N\,:\,\sup_{s\le h} \tilde{I}_s^i \ge c_2\varepsilon - c_1h - a(1 + c_1)\} \ge N\eta/4 \right] +\prob\left[\frac{1}{N}\sum_{j = 1}^N \sup_{s \le t + h} \nnnorm{Z_{s}^j} \ge \frac{a}{2h}\right] \\
		&+ \prob\left[\sup_{s \le h}\nnnorm{I_s} > a\right]. 
	\end{align*}
	By Markov's inequality then Burkholder-Davis-Gundy,
	\begin{equation*}
		\prob\left[\sup_{s \le h}\nnnorm{I_s} > a\right] \le a^{-2}\E\left[\sup_{s \le h}I_s^2\right] \lesssim ha^{-2}.
	\end{equation*}
	Also, by Markov's inequality and uniform sub-Gaussianity of $Z^{i}$, we have
	\begin{equation*}
		\prob\left[\frac{1}{N}\sum_{j = 1}^N \sup_{s \le t + h} \nnnorm{Z_{s}^j} \ge \frac{a}{2h}\right] \lesssim ha^{-1},
	\end{equation*}
	with a constant that holds uniformly in $N$. Hence, as $\tilde{I}^i$ is a time-changed Brownian motion, we have 
	\begin{align*}
		\prob\left[E\right] 
		&\le \prob\left[\#\{i \le N\,:\,\sup_{s\le h} {B}_s^i \ge c_2\varepsilon - c_1h - a(1 + c_1)\} \ge N\eta/4 \right] + O(ha^{-2} + ha^{-1})\\
		&+ \sum_{{\mathcal{I}_0^{\textnormal{ref}} \subset \{1,\ldots,N\}}} \prob\left[\left.\#\{i \in \mathcal{I}_0^{\textnormal{ref}}\,:\,\sup_{\hat{\tau}^i\le s\le t + h}\left\{X_s^{i} - X_{\hat{\tau}^i}^{i}\right\} \ge \varepsilon\} \ge N\eta/4 \right|\mathcal{I}^{\textnormal{ref}} = \mathcal{I}_0^{\textnormal{ref}}\right]\\
		&\hspace{2.5cm}\times\prob\left[\mathcal{I}^{\textnormal{ref}} = \mathcal{I}_0^{\textnormal{ref}}\right]\\
		&= \prob\left[\frac{1}{N} \sum_{i = 1}^N\ind_{\{\sup_{s\le h} {B}_s^i \ge c_2\varepsilon - c_1h - a(1 + c_1)\}} \ge \eta/4 \right] + O(ha^{-2} + ha^{-1})\\
		&+ \sum_{{\mathcal{I}_0^{\textnormal{ref}} \subset \{1,\ldots,N\}}} \prob\left[\left.\#\{i \in \mathcal{I}_0^{\textnormal{ref}}\,:\,\sup_{\hat{\tau}^i\le s\le t + h}\left\{X_s^{i} - X_{\hat{\tau}^i}^{i}\right\} \ge \varepsilon\} \ge N\eta/4 \right|\mathcal{I}^{\textnormal{ref}} = \mathcal{I}_0^{\textnormal{ref}}\right]\\
		&\hspace{2.5cm}\times\prob\left[\mathcal{I}^{\textnormal{ref}} = \mathcal{I}_0^{\textnormal{ref}}\right]
	\end{align*}
	We now turn our attention to the second term on the R.H.S. of \eqref{eq: THE MAIN OBJECT OR PROBABILITY WE LOOK TO HAVE A CONTROL ON IN H IN THE PROOF THAT WE HAVE THE FIRING FUNCTION BEING TIGHT 3}. We note that the same argument as above can be applied to this term. Therefore, we have
	\begin{align*}
		\prob\left[E\right] 
		\le& \prob\left[\frac{1}{N} \sum_{i = 1}^N\ind_{\{\sup_{s\le h} {B}_s^i \ge c_2\varepsilon - c_1h - a(1 + c_1)\}} \ge \eta/4 \right] + O(ha^{-2} + ha^{-1})\\
		&+ \prob\left[\frac{1}{N} \sum_{i = 1}^N\ind_{\{\sup_{s\le h} {B}_s^i \ge c_2\varepsilon - c_1h - a(1 + c_1)\}} \ge \eta/4 \right]
	\end{align*}
	Then, the law of large numbers and the distribution of the maximum of a Brownian motion gives 
	\begin{equation*}
		\limsup_{N \to \infty} \prob\left[E\right] \le 2\ind_{\{2\phi(-c_2\varepsilon h^{-1/2}+ c_1h^{1/2} + ah^{-1/2}(1 + c_1)) \ge \eta/4\}} + O(ha^{-2} + ha^{-1})
	\end{equation*}
	provided $c_2\varepsilon - c_1h - a(1 + c_1) > 0 $, where $\phi$ is the normal cdf. Making the choice
	\begin{equation*}
		\varepsilon(h) = h^{1/2}\log(1/h) \quad \text{and} \quad a(h) = h^{1/2}\log\log(1/h)
	\end{equation*}
	ensures that 
	\begin{equation*}
		\varepsilon(h) \to  0, \qquad ha^{-1} \to  0, \qquad ha^{-2} \to  0,
	\end{equation*}
	and
	\begin{equation*}
		-c_2\varepsilon h^{-1/2}+ c_1h^{1/2} + ah^{-1/2}(1 + c_1) = -c_2\log(1/h) + c_1h^{1/2} + \log\log(1/h)(1 + c_1) \to  -\infty
	\end{equation*}
	as $h \to  0$.
	
	\underline{Step 3:}
	Bringing it all together, we have shown that, for $h$ small enough,
	\begin{align*}
		&\limsup_{N \to  \infty} \prob\left[F_{t+h}^N - F_t^N \ge 2\eta\right] \\
		&\hspace{3cm}\le \limsup_{N \to  \infty} \prob\left[F_{t+h}^N - F_t^N \ge 2\eta, \nnnorm{\mathcal{I}^{(2)}} < N\beta\right] + \prob\left[\nnnorm{\mathcal{I}^{(2)}} \ge N\beta\right]\\
		&\hspace{3cm}\lesssim {h^{1/8}}{\eta^{-1}} +2\ind_{\{2\phi(-c_2\varepsilon h^{-1/2}+ c_1h^{1/2} + ah^{-1/2}(1 + c_1)) \ge \eta/4\}} + ha^{-2} + ha^{-1}\\
		&\hspace{3cm}+{4}{\eta^{-1}}\limsup_{N \to  \infty}\E\nu_t^N(-\varepsilon,\,0) + h,
	\end{align*}
	where the right-hand side tends to $0$ as $h$ tends to $0$. The control on the left increment $F_{t}^N - F^N_{t-h}$ follows in the same way, so the proof is complete.
\end{proof}

\section{Regularity Properties of Limit Points}

Let $\mathscr{S}$ denote the space of Schwartz functions on $\mathbb{R}$, that is, the space of rapidly decreasing infinitely differentiable functions, and let $\mathscr{S}^\prime$ denote its dual, the space of tempered distributions. We write $(D_{\mathscr{S}^\prime},M_1)$ for the space of $\mathscr{S}^\prime$-valued \cadlag processes on $[0,T]$, equipped with the $M_1$-topology as constructed in \citep{ledger2016skorokhod}. The $M_1$-topology is useful for our setting as monotone real-valued functions are automatically tight in this topology. With this in mind, and with some work, one can show that the empirical measure of the particle system is tight in this space and has a limit point in $(D_{\mathscr{S}^\prime},M_1)$. Hence we can use the properties and characteristics of the particle system to deduce properties of the limiting one, which is done in the two propositions below.

\begin{proposition}\label{lem: BOUNDS ON THE TAIL MOMENTS OF THE EMPIRICAL MEASURES AND THE LIMITNG MEASURE}
	Suppose $\nu^N$ converges weakly towards $\nu^*$ on $(D_{\mathscr{S}^\prime},M_1)$. Then, for every $a > 0$, we have uniformly in $N \ge 1$ and $t \in [0,T]$ that 
	\begin{align}
		\E{\langle\nu_t^N,\nnorm{x}^k\1_{(-\infty,-\lambda]}(x)\rangle} &= o(\lambda^ke^{-a\lambda}), \label{eq:tail_moment_pointwise}\\
		\E{\int_0^T\langle\nu_t^N,\nnorm{x}^k\1_{(-\infty,-\lambda]}(x)\rangle\diff t} &= o(\lambda^ke^{-a\lambda}), \label{eq:tail_moment_integrated_empirical}\\
		\E{\int_0^T\langle\nu_t^*,\nnorm{x}^k\1_{(-\infty,-\lambda]}(x)\rangle\diff t} &= o(\lambda^ke^{-a\lambda}). \label{eq:tail_moment_integrated_limit}
	\end{align}
\end{proposition}

\begin{proof}
	We are in an analogous setting to that of \citep{hambly2019spde}. The proof follows along the same lines as \citep[Lemma~A.3]{hambly2019spde} with the obvious changes.
\end{proof}

\begin{proposition}\label{lem: LEMMA ON CADLAG-NESS OF LIMIT POINT AND BOUNDNESS OF THE SUP}
	Suppose $\nu^N$ converges weakly towards $\nu^*$ on $(D_{\mathscr{S}^\prime},M_1)$. Then for any polynomial $f$, we have that $t \mapsto \langle \nu^*_t, f\rangle$ has paths in $D_\R$, the space of $\R$-valued \cadlag processes, almost surely. Furthermore, $\E\sup_{t \le T} \langle\nu_t^*,\nnnorm{x}^p\ind_{\{\nnnorm{x} > \lambda\}}\rangle = o(e^{-a\lambda})$ as $\lambda \to  \infty$, for any $a > 0$.
\end{proposition}
\begin{proof}
	Without loss of generality, it suffices to show the claim for $f(x) = x^p$ for some $p \in \N$. To begin, for $R,\,n \in \N$ we choose $\chi_R,\phi_n \in \mathcal{C}^{\infty}(\mathbb{R};[0,1])$ such that
	\begin{equation*}
		\chi_R = \begin{cases}
			1 \qquad &\textrm{ on } [-R,R],\\
			0 \qquad &\textrm{ on } (-R-1,R + 1)^\complement,
		\end{cases}
		\quad \textnormal{and} \quad \phi_n =\begin{cases}
			1 \qquad &\textrm{ on } [-2n,-R-1/n],\\
			0 \qquad &\textrm{ on } (-2n-1,-R)^\complement.
		\end{cases}
	\end{equation*}
	Here, $\chi_R$ is a standard cut-off function and $\phi_n$ approximates $\ind_{\{\nnnorm{x} > R\}}$ pointwise. Then for any $s,t \in [0,T]$,
	\begin{align*}
		\nnnorm{\langle\nu_t^*,x^p\rangle  -\langle\nu_s^*,x^p\rangle} 
		\le& \nnnorm{\langle\nu_t^*,x^p\rangle  -\langle\nu_t^*,x^p\chi_R\rangle} 
		+ \nnnorm{\langle\nu_s^*,x^p\rangle  -\langle\nu_s^*,x^p\chi_R\rangle} \\
		&+ \nnnorm{\langle\nu_t^*,x^p\chi_R\rangle  -\langle\nu_s^*,x^p\chi_R\rangle} \\
		=& \Romannum{1} + \Romannum{2} + \Romannum{3}.
	\end{align*}
	The strategy will be to show that with probability one, we can choose a (stochastic) $R$, independent of $t$ and $s$, such that we may make $\Romannum{1}$ and $\Romannum{2}$ as small as we like. Therefore, as $t \mapsto \langle\nu_t^*,x^p\chi_R\rangle$ has \cadlag paths, the claim follows. By definition of $\chi_R$, $\nnnorm{\langle\nu_t^*,x^p\rangle  -\langle\nu_t^*,x^p\chi_R\rangle} \le 2 \langle\nu_t^*,\nnnorm{x}^p\ind_{\{\nnnorm{x} > R\}}\rangle$. Therefore, we aim to gain control over $\sup_{t \le T} \langle\nu_t^*,\nnnorm{x}^p\ind_{\{\nnnorm{x} > R\}}\rangle$. To this end, by the definition of $\phi_n$, we have $\lim_{n \to  \infty} \sup_{t \le T} \langle\nu_t^*,\nnnorm{x}^p\phi_n\rangle = \sup_{t \le T} \langle\nu_t^*,\nnnorm{x}^p\ind_{\{\nnnorm{x} > R\}}\rangle$. Therefore, by applying Fatou's Lemma and then the Portmanteau Theorem,
	\begin{equation}\label{eq: FIRST EQUATION OF INTEREST IN THE PROOF THAT INTEGRAL OF NU STAR AGAINST A POLYNOMIAL HAS CADLAG PATHS}
		\E\sup_{t \le T} \langle\nu_t^*,\nnnorm{x}^p\ind_{\{\nnnorm{x} > R\}}\rangle \le \liminf_{n \to  \infty} \E\sup_{t \le T} \langle\nu_t^*,\nnnorm{x}^p\phi_n\rangle  = \liminf_{n \to  \infty} \limsup_{N \to  \infty} \E\sup_{t \le T} \langle\nu_t^N,\nnnorm{x}^p\phi_n\rangle.
	\end{equation}
	Furthermore, we may bound the last term in \eqref{eq: FIRST EQUATION OF INTEREST IN THE PROOF THAT INTEGRAL OF NU STAR AGAINST A POLYNOMIAL HAS CADLAG PATHS} by
	\begin{align*}
		\E\sup_{t \le T} \langle\nu_t^N,\nnnorm{x}^p\phi_n\rangle 
		&\le \E \left[\sup_{t \le T} N^{-1} \sum_{i =1}^N\nnnorm{X_t^i}^p\ind_{\{\nnnorm{X_t^i} > R\}}\right] \\
		&\le N^{-1} \sum_{i =1}^N \E \left[\sup_{t \le T} \nnnorm{X_t^i}^p\ind_{\{\sup_{t \le T} \nnnorm{X_t^i} > R\}}\right],
	\end{align*}
	where the first inequality follows from the definition of $\nu^N$ and $\phi_n$. By \cref{cor: SUBGUASSIANTY COROLLARY}, $\E[\sup_{t \le T} \nnnorm{X_t^i}^p]$ is uniformly bounded in $i$ and $N$ for every $p$ (fixed). Therefore, by applying H\"older's inequality to the final expression above and then employing \cref{cor: SUBGUASSIANTY COROLLARY}, we deduce there exist $C,\delta > 0$, independent of $N$, $R$ and $i$, such that $\E[\sup_{t \le T} \nnnorm{X_t^i}^p\ind_{\{\sup_{t \le T} \nnnorm{X_t^i} > R\}}] \le C\exp\{-\delta R^2\}$. Hence, we have shown $\E\sup_{t \le T} \langle\nu_t^*,\nnnorm{x}^p\ind_{\{\nnnorm{x} > R\}}\rangle \le C\exp\{-\delta R^2\}$. This proves the second claim. By the Borel-Cantelli Lemma, 
	\begin{equation*}
		\prob\left[\sup_{t \le T} \langle\nu_t^*,\nnnorm{x}^p\ind_{\{\nnnorm{x} > R\}}\rangle > \exp\{-\delta R^2/2\} \quad \textnormal{i.o.}\right] = 0.
	\end{equation*}
	Hence, the set 
	\begin{equation*}
		\Omega^* \coloneqq \bigcup_{R\in \N} \bigcap_{\substack{\tilde{R} \ge R\\ \tilde{R} \in \N}}\left\{\sup_{t \le T} \langle\nu_t^*,\nnnorm{x}^p\ind_{\{\nnnorm{x} > \tilde{R}\}}\rangle \le \exp\{-\delta \tilde{R}^2/2\}\right\}
	\end{equation*}
	has full measure. Now, for any $\omega \in \Omega^*$ and $\varepsilon > 0$, we may choose an $R$ large enough such that $\Romannum{1} + \Romannum{2} < \varepsilon$. Furthermore, for any fixed $R$, $t \mapsto \langle\nu_t^*,x^p\chi_R\rangle$ has \cadlag paths. Combining these two facts, it is routine to show that $t \mapsto \langle \nu^*_t, x^p\rangle (\omega)$ is \cadlag.
\end{proof}

\printbibliography

\end{document}